\documentclass[11pt,notitlepage,a4paper]{article}

\usepackage{amssymb} 
\usepackage[intlimits]{amsmath}
\usepackage{amsfonts}
\usepackage{amsthm} 

\usepackage{mathptmx}

\usepackage{hyperref}
\hypersetup{colorlinks=true,linkcolor=RoyalPurple,citecolor=RoyalPurple}

\usepackage{cite}

\usepackage{graphicx}

\usepackage{geometry}

\usepackage{color}
\usepackage[dvipsnames]{xcolor}

\let\oldsqrt\sqrt
\def\sqrt{\mathpalette\DHLhksqrt}
\def\DHLhksqrt#1#2{%
\setbox0=\hbox{$#1\oldsqrt{#2\,}$}\dimen0=\ht0
\advance\dimen0-0.2\ht0
\setbox2=\hbox{\vrule height\ht0 depth -\dimen0}%
{\box0\lower0.4pt\box2}}

\allowdisplaybreaks

\newcommand{\R}{\mathbb{R}} 
\newcommand{\N}{\mathbb{N}} 
\newcommand{\Z}{\mathbb{Z}} 

\newcommand{\dist}{\textnormal{dist}} 
\newcommand{\diam}{\textnormal{diam}} 
\newcommand{\supp}{\textnormal{supp}} 
\newcommand{\essinf}{\textnormal{essinf}} 
\newcommand{\esssup}{\textnormal{esssup}} 

\DeclareMathOperator{\vol}{vol}

\renewcommand{\phi}{\varphi}

\newcommand{\cC}{{\mathcal C}}
\newcommand{\cD}{{\mathcal D}}
\newcommand{\cE}{{\mathcal E}}

\newcommand{\cH}{{\mathcal H}}

\newcommand{\cV}{{\mathcal V}}
\newcommand{\cW}{{\mathcal W}}

\newcommand{\eps}{\varepsilon}

\theoremstyle{definition}
\theoremstyle{plain} 
\newtheorem{defi}{Definition}[section]
\newtheorem{thm}[defi]{Theorem}
\newtheorem{prop}[defi]{Proposition}
\newtheorem{lemma}[defi]{Lemma}
\newtheorem{cor}[defi]{Corollary}

\theoremstyle{definition}

\numberwithin{equation}{section} 

\title{On the strong maximum principle for nonlocal operators}
\author{
 \ Sven Jarohs\footnote{Institut f\"ur Mathematik, Goethe-Universit\"at, Frankfurt, Robert-Mayer-Stra\ss e 10, D-60629 Frankfurt, jarohs@math.uni-frankfurt.de.}
 \;   and
 \!\ Tobias Weth\footnote{Institut f\"ur Mathematik, Goethe-Universit\"at, Frankfurt, Robert-Mayer-Stra\ss e 10, D-60629 Frankfurt, weth@math.uni-frankfurt.de.}}
\date{\today}

\begin{document}
\maketitle

\begin{abstract}
In this paper we derive a strong maximum principle for weak supersolutions of nonlocal equations of the form 
$$
		Iu=c(x) u \qquad \text{ in $\Omega$,}
$$
where $\Omega\subset \R^N$ is a domain, $c\in L^{\infty}(\Omega)$ and $I$ is an operator of the form
\[
Iu(x)=P.V.\int_{\R^N}(u(x)-u(y))j(x-y)\ dy
\]
with a nonnegative kernel function $j$. We formulate minimal positivity assumptions on $j$ corresponding to a class of operators, which includes highly anisotropic variants of the fractional Laplacian. Somewhat surprisingly, this problem leads to the study of general lattices in $\R^N$. Our results extend to the regional variant of the operator $I$ and, under weak additional assumptions, also to the case of $x$-dependent kernel functions.  

\end{abstract}
{\footnotesize
\begin{center}
\textit{Keywords.} Nonlocal Operator $\cdot$ Strong Maximum Principle $\cdot$ Weak Maximum Principle
\end{center}
\begin{center}
\end{center}
}

\section{Introduction}\label{sec:introduction}
In the study of elliptic partial differential equations of second order, one of the most important tools are maximum principles of weak and strong type, as they are intimitely related to the theory of existence, regularity and symmetry of solutions. 
One of the simplest equations, where maximum principles arise, are linear elliptic second order PDEs of the type 
\begin{equation}
  \label{eq:second-order}
-\Delta u =c(x) u \qquad \text{in $\Omega$.}
\end{equation}
Here, $\Omega \subset \R^N$ is a domain and $c \in L^\infty_{loc}(\Omega)$. In this case, a strong maximum principle can be stated as follows: If $u \in H^1_{loc}(\Omega)$, $u \ge 0$ is a weak nontrivial supersolution of \eqref{eq:second-order}, i.e.,
\begin{equation}
  \label{eq:second-order-weak-supersol}
\int_{\Omega} \nabla u \nabla \phi\,dx \ge \int_{\Omega} c(x) u \phi \,dx \qquad \text{for all $\phi \in C_c^\infty(\Omega)$, $\phi \ge 0$,}
\end{equation}
then $u$ is strictly positive in $\Omega$. Here $H^1_{loc}(\Omega)$ denotes the standard local first order Sobolev space. We note that $u$ does not necessarily need to be continuous in $\Omega$; one may consider e.g. $\Omega= B_1(0) \subset \R^N$ and 
the function $x \mapsto -\ln |x|$ which is contained in $H^1(\Omega)$, nonnegative and  weakly superharmonic in $\Omega$ if $N \ge 3$. Hence the strict positivity should be understood in the sense that 
\begin{equation}
  \label{eq:def-strictly-positive}
\underset{K}\essinf \: u >0 \qquad \text{for every compact subset $K \subset \Omega$.}
\end{equation}
For functions satisfying (\ref{eq:second-order-weak-supersol}), the strict positivity follows from the classical Harnack inequality (see e.g. \cite[Theorem 8.18]{GT}), which can be seen as quantitative version of the strong maximum principle. In the present paper we are concerned with maximum principles for weak supersolutions to equations of type
\begin{equation}
  \label{nonlocal-general-0}
		Iu =c(x) u \qquad \text{in $\Omega$,}
\end{equation}
where $c \in L^\infty_{loc}(\Omega)$ and $I$ is a nonlocal operator formally given by 
\begin{equation}
  \label{eq:def-Iu-formal}
Iu(x)=P.V.\int_{\R^N}(u(x)-u(y))j(x-y)\ dy : = \lim_{\eps \to 0^+} 
\int_{\R \setminus B_{\eps}(x)}(u(x)-u(y))j(x-y)\ dy.
\end{equation}
Here, $j: \R^N \to [0,\infty]$ is the associated (nonnegative) kernel function, which typically has a singularity at the origin. Our aim is to formulate optimal conditions on $j$ such that a strong maximum principle holds for weak supersolutions of (\ref{nonlocal-general-0}). It turns out that we only need the following two assumptions. 
\begin{itemize}
\item[(j\,\!1)] {\em (L\'evy type integrability condition)} The kernel $j: \R^N \to [0,\infty]$ is even and measurable with\footnote{Here and in the following we use the notation $a\wedge b=\min\{a,b\}$ for $a,b\in \R$.}
$$
\int \limits_{\R^N} 1 \wedge |z|^2j(z)\, dz <\infty.
$$
\item[(j\,\!2)] {\em (Nontriviality condition)} For every $r>0$, $j$ does not vanish a.e. in $B_r(0)$.
\end{itemize}
In order to motivate these assumptions, let us assume for the moment that $j$ satisfies (j\,\!1), (j\,\!2) and has finite total mass, i.e., 
\begin{equation}
  \label{eq:assumption-finite}
\int_{\R^N}j(z)\,dz < \infty. 
\end{equation}
Moreover, let us consider a nonnegative bounded pointwise supersolution $u:\R^N \to \R$ of equation (\ref{nonlocal-general-0}) in a domain $\Omega \subset \R^N$, 
which then satisfies
\begin{equation}
\label{nonlocal-general-0-pointwise-bounded}  
u(x)\Bigl(c^-(x) + \int_{\R^N}j(z)\ dz \Bigr) \ge \int_{\R^N}u(x+z)j(z)\,dz, \qquad x \in \Omega.
\end{equation}
Here $c^-:= - \min \{c,0\}$ denotes the negative part of $c$. 
Since $j$ is nontrivial by assumption (j\,\!2), $u$ is positive in $x \in \Omega$ iff $\int_{\R^N}u(x+z)j(z)\,dy$ is positive. If we assume in addition that $j$ is strictly positive in $B_r(0)$ for some $r>0$, a continuation argument shows that either $u \equiv 0$ in $\Omega$ or $u > 0$ in $\Omega$. Hence the strong maximum principle readily follows in this case.  For a related result in this context, see e.g. \cite[Theorem 7]{garcia-rossi}. If we merely assume condition (j\,\!2) in place of the strict positivity of $j$ in a neighborhood of zero, the same conclusion is much less clear. We shall see in this paper that indeed a continuation argument can be performed along suitably chosen lattice paths in $\R^N$.

Although it is instructive to consider the case where \eqref{eq:assumption-finite} holds, the equation (\ref{nonlocal-general-0}) becomes much more interesting in the case where $\int_{\R^N}j(z)\,dz = \infty$. In this case, the pointwise inequality (\ref{nonlocal-general-0-pointwise-bounded}) makes no sense anymore, and it is appropriate to consider weak supersolutions of (\ref{nonlocal-general-0}) instead. We note that assumptions (j\,\!1)--(j\,\!2) include different variants of the fractional Laplacian and more general symmetric L\'evy type operators with absolutely continuous L\'evy measure and with vanishing diffusion and drift coefficients, see e.g. \cite[Section 3.3.2]{A09}, \cite[Chapters II. 2 and III.7]{J05} and \cite{fall-weth}. We will discuss these and other examples in detail below.

In order to define the notion of a weak supersolution of (\ref{nonlocal-general-0}), we consider the 
associated bilinear form
\begin{equation}\label{form-j}
(u,v) \mapsto \cE_{j}(u,v)=\frac{1}{2}\int_{\R^N}\int_{\R^N}(u(x)-u(y))(v(x)-v(y))j(x-y)\ dxdy
\end{equation}
and the function space 
\begin{equation}
\label{def-V-J-loc}
\cV^j_{loc}(\Omega):=\Bigl\{u\in L^2_{loc}(\R^N) :\! \int_{\Omega'}\int_{\R^N} (u(x)-u(y))^2j(x-y)\ dxdy<\infty \; \text{for every $\Omega' \subset \subset \Omega$}\Bigr\}.
\end{equation}
Here and in the following, $\Omega' \subset \subset \Omega$ means that $\overline {\Omega'}$ is compact and contained in $\Omega$. The use of this space is inspired by \cite{DK16,FKV13}. We remark that $\cV^j_{loc}(\Omega)$ contains all functions, which are bounded on $\R^N$ and locally Lipschitz in $\Omega$, see Lemma~\ref{Lipschitz-V-J} below. Moreover, in the special case where $j(z)=|z|^{-N-2s}$ with some $s \in (0,1)$, the space $\cV^j_{loc}(\Omega)$ contains all functions $u \in L^2_{loc}(\R^N)$ with 
$$
\int_{\R^N} \frac{u^2(x)}{1+|x|^{N+2s}}\,dx < \infty \qquad \text{and}\qquad u\big|_{\Omega} \in H^s_{loc}(\Omega),
$$
where $H^s_{loc}(\Omega)$ is the usual local Sobolev space of order $s$, see e.g. \cite{grisvard}. We call a function $u \in \cV^j_{loc}(\Omega)$ a weak supersolution of (\ref{nonlocal-general-0}) in $\Omega$, if 
 \begin{equation}
   \label{eq:assumption-supersol-variant-1}
\cE_{j}(u,\phi) \ge \int_{\Omega} c(x) u \phi \,dx \qquad \text{for all $\phi \in C^\infty_c(\Omega)$, $\:\phi \ge 0$.}
 \end{equation}
We shall see in Section~\ref{notation} that $\cE_j(u,\phi)$ is indeed well-defined (in the Lebesgue sense) if $u \in \cV^j_{loc}(\Omega)$ and $\phi \in C^{\infty}_c(\Omega)$. Our main result is the following. 

	\begin{thm}\label{main-theorem} {(Strong maximum principle)}\\[0.1cm]
Suppose (j\,\!1) and (j\,\!2), let $\Omega\subset\R^{N}$ be a domain, and let $c \in L^\infty_{loc}(\Omega)$. Moreover, let $u \in \cV^j_{loc}(\Omega)$ be a weak supersolution of (\ref{nonlocal-general-0}) in $\Omega$ with $u \ge 0$ on $\R^N$.\\[0.2cm]
Then either $u \equiv 0$ in $\Omega$, or $u$ is strictly positive in $\Omega$.
	\end{thm}
Here and in the following, a measurable function $u$ on $\Omega$ will be called strictly positive in $\Omega$ if (\ref{eq:def-strictly-positive}) holds. As indicated already, a notable feature of Theorem~\ref{main-theorem} is the weak positivity assumption (j\,\!2) which allows to consider highly anisotropic kernels $j$. We illuminate this aspect by discussing kernels of the form
\begin{equation}
  \label{eq:example-j}
z \mapsto j(z)= 1_{A}(z) |z|^{\tau},
\end{equation}
where $1_A$ is the characteristic function of a symmetric measurable subset $A \subset \R^N$ and $\tau \in \R$. In this case, assumption (j\,\!2) amounts to the condition 
\begin{equation}
  \label{eq:j-2-example}
|A \cap B_r(0)| >0 \qquad \text{for every $r>0$,}  
\end{equation}
where $|\cdot|$ stands for Lebesgue measure. Moreover, for this type of kernels,
we may write condition (j\,\!1) as follows: 
\begin{equation}
  \label{eq:j-1-example}
\int_0^1 r^{\tau + 2} \vol_{N-1}(A \cap S_r)\,dr + 
\int_1^\infty r^{\tau} \vol_{N-1}(A \cap S_r)\,dr < \infty,
\end{equation}
Here $S_r$ denotes the sphere of radius $r>0$ centered at $0$. In the case where $A = \R^N$ and 
\begin{equation}
  \label{eq:assumption-tau-s}
\text{$\tau = -N -2s\quad$ for some $s \in (0,1)$,}  
\end{equation}
the operator $I$ coincides up to a constant with the fractional Laplacian $(-\Delta)^s$, which has been studied extensively in recent years.
A probabilistic argument yielding the strong maximum principle for supersolutions of the equation $(-\Delta)^s u= c(x) u$ can be found in \cite[p.312--313]{BB00}. Moreover, in \cite[Proposition 2.7]{CRS10}) and \cite[Section 4.6]{CS14} the representation of $(-\Delta)^s$ as a Dirichlet-to-Neumann type operator is used to derive a strong maximum principle.
\\
Anisotropic versions of the fractional Laplacian arise when considering  (\ref{eq:example-j}) and (\ref{eq:assumption-tau-s}) for a general symmetric measurable subset $A \subset \R^N$ satisfying 
\begin{equation}
  \label{eq:two-sided-ineq}
\vol_{N-1}(A \cap S_r) \ge c r^{N-1} \qquad \text{for all $r>0$ with a constant $c>0$.}
\end{equation}
Operators $I$ of this type are considered in the recent paper \cite{DK15} by Dyda and Kassmann, who proved, under additional assumptions, a weak Harnack inequality which implies the strong maximum principle, see \cite[Theorem 1.1]{DK15}. We also mention \cite{FK12} where a parabolic Harnack inequality has been proved in this context. The lower bound in (\ref{eq:two-sided-ineq}) is of key importance for a Harnack inequality to hold. In contrast, it is not required in Theorem~\ref{main-theorem}, where arbitrary measurable symmetric sets $A \subset \R^N$ satisfying (\ref{eq:j-2-example}) may be considered together with exponents of the form (\ref{eq:assumption-tau-s}). In the case where $A$ is bounded, we may relax assumption~(\ref{eq:assumption-tau-s}) and consider any $\tau > -N-2$. 

In the special case where $\tau= -N$ and $A$ is a bounded symmetric set containing a small ball $B_r(0)$ for some $r>0$, the operator $I$ is a zero order operator, which still has a regularizing effect. More precisely, in this case it has been observed in the recent work by Kassmann and Mimica \cite{KM13} that solutions of $Iu = f$ are continuous if $f$ is bounded. It is a challenging open question whether such a regularity result still holds for $\tau= -N$ and general bounded and symmetric $A$ satisfying only (\ref{eq:j-2-example}). Theorem \ref{main-theorem} shows that the strong maximum principle holds for this class of operators and weak supersolutions. 
  
Another interesting aspect is given by the fact that, depending on the shape of $A$, $j$ is allowed to have a singularity of arbitrarily high order. For example, if 
 $$
 A:= \Bigl \{(x_1,x') \in [-1,1] \times \R^{N-1}\::\: |x'| \le |x_1|^\rho \Bigr\} \subset \R^N
 $$
for some $\rho > 1$ and $\tau > -3-(N-1)\rho$, then (\ref{eq:j-2-example}) and (\ref{eq:j-1-example}) are satisfied while \eqref{eq:two-sided-ineq} does not hold (see also \cite[Example 6]{DK15} for a related example). \\

The strong maximum principle given in Theorem~\ref{main-theorem} can be extended to {\em regional operators} of the form 
\[
I_{\Omega} u(x)=P.V.\int_{\Omega}(u(x)-u(y))j(x-y)\ dy : = \lim_{\eps \to 0^+} 
\int_{\Omega \setminus B_{\eps}(x)}(u(x)-u(y))j(x-y)\ dy
\]
under the same assumptions (j1) and (j2). In order to define a notion of weak supersolutions of the equation 
\begin{equation}
  \label{nonlocal-general-0-regional}
		I_\Omega u =c(x) u \qquad \text{in $\Omega$,}
\end{equation}
we introduce the space 
\begin{equation}
\label{def-H-J-loc}
\cH^j_{loc}(\Omega):= \Bigl\{u \in L^2_{loc}(\overline{\Omega}):\! \int_{\Omega'}\int_{\Omega}
(u(x)-u(y))^2 j(x-y)\,dxdy < \infty\; \text{for every $\Omega'\subset \subset \Omega$}\Bigr\}.
\end{equation}
The following theorem then arises as a rather direct corollary of Theorem~\ref{main-theorem}.

	\begin{thm}\label{main-theorem-regional} 
Suppose that (j\,\!1) and (j\,\!2) hold, let $\Omega\subset\R^{N}$ be a domain, and let $c \in L^\infty_{loc}(\Omega)$.
Moreover, let $u \in \cH^j_{loc}(\Omega)$ be a weak supersolution of (\ref{nonlocal-general-0-regional}), i.e., 
 \begin{equation*}
\int_{\Omega}\int_{\Omega}
(u(x)-u(y))(\phi(x)-\phi(y)) j(x-y)\,dxdy \ge \int_{\Omega} c(x) u \phi \,dx\qquad \text{for all $\phi \in C^\infty_c(\Omega)$, $\phi \ge 0$.} 
 \end{equation*}
Suppose furthermore that $u \ge 0$ a.e. in $\Omega$.\\[0.1cm]
Then either $u \equiv 0$ in $\Omega$, or $u$ is strictly positive in $\Omega$.
	\end{thm}

To deduce this result from Theorem~\ref{main-theorem}, we merely note that, under the assumptions of Theorem~\ref{main-theorem-regional}, the trivial extension $\tilde u$ of $u$ to $\R^N$ is contained in $\cV^j_{loc}(\Omega)$, and it is a nonnegative weak supersolution of the equation $I u = \tilde c(x)u$ with 
$$
\tilde c \in L^\infty_{loc}(\Omega),\qquad \tilde c(x)= c(x) + \int_{\R^N \setminus \Omega}j(x-y)\,dy.
$$
Hence Theorem~\ref{main-theorem} shows that $\tilde u$ is strictly positive in $\Omega$, so the conclusion of Theorem~\ref{main-theorem-regional} follows.
 
We note that, in the special case where $j(z)= |z|^{-N-2s}$ for some $s \in (0,1)$, the operator $I_\Omega$ is called the {\em regional fractional Laplacian}, see e.g. \cite{QM05}. A strong maximum principle in this special case has been given in \cite[Theorem 4.1]{MN17}.

The strategy of the proof of Theorem~\ref{main-theorem} consists in 
successively increasing the region of positivity of nontrivial nonnegative supersolutions of (\ref{nonlocal-general-0}). This can be done by means of a  weak maximum principle for domains with small volume and suitably constructed comparison functions. This weak maximum principle needs to be derived in a preliminary step based on assumption (j\,\!1). In the present framework, a small volume weak maximum principle can be stated as follows. 

	\begin{thm}\label{weak-max-intro}
Suppose that (j\,\!1) is satisfied, and let $c \in L^\infty(\R^N)$ satisfy 
$$
\|c^+\|_{L^\infty(\R^N)}< \int_{\R^N}j(z)\,dz\in(0,\infty]. 
$$
Then there exists $r>0$ such that for every open set $\Omega \subset \R^N$ with $|\Omega| \le r$ and any weak supersolution $u \in \cV^j_{loc}(\Omega)$ of (\ref{nonlocal-general-0}) in $\Omega$ such that  
\begin{equation}
  \label{eq:weak-max-intro-cond}
\text{$u$ is nonnegative outside a compact subset of $\Omega$}
\end{equation}
we have $u \ge 0$ in $\R^N$.
\end{thm}

Related weak maximum principles have been derived e.g. in \cite{jarohs-thesis,JW14,FKV13} under different assumptions. The proof of Theorem~\ref{weak-max-intro} is inspired  by \cite{jarohs-thesis,JW14} and relies on estimates for the bilinear form $\cE_j$ and for the small volume asymptotics of the first Dirichlet eigenvalue $\Lambda_1(\Omega)$ of the operator $I$, see (\ref{firsteigen}) and (\ref{blowup}) below. In addition to these estimates, we also need a density property since our notion of supersolutions is based on testing only with $C^\infty_c(\Omega)$-functions, see Proposition~\ref{density} below. In fact, in the proof of Theorem~\ref{main-theorem}, it will be convenient to introduce a stronger notion of supersolutions with respect to a larger Hilbertian test function space related to the variational features of the problem. We shall do this in Section \ref{sec:main-results-general}, where we also formulate weak and strong maximum principles 
in a more general framework of bilinear forms  
\begin{equation}\label{J-form}
(u,v) \mapsto \cE_J(u,v)=\frac{1}{2}\int_{\R^N}\int_{\R^N}(u(x)-u(y))(v(x)-v(y))J(x,y)\ dxdy
\end{equation}
with $x$-dependent kernel functions $J(x,y)$.
 
The most difficult step in the proof of Theorem~\ref{main-theorem} is to prove that any nontrivial nonnegative supersolution $u \in \cV^j_{loc}(\Omega)$ of (\ref{nonlocal-general-0}) is strictly positive on larger and larger subsets of $\Omega$. Within this step, we first introduce the notion of the {\em subsolution property (SSP)} of a pair $(K,M)$ of subsets of $\Omega$. By the weak maximum principle, this property turns out to be a sufficient criterion for a weak supersolution of (\ref{nonlocal-general-0}) to inherit strict positivity on $M$ from uniform positivity on $K$. Then, using the local nontriviality condition (j2) for the kernel, we determine finite sequences $(K_i,M_{i+1})_i$ of pairs of subsets of $\Omega$ satisfying $(SSP)$ and such that we can successively prove that $u$ is strictly positive in $M_{i}$ for all $i$. Somewhat surprisingly, in this step we are led to prove a purely geometric existence result for localized paths in general lattices $\sum \limits_{k=1}^N v_i \Z \subset \R^N$ generated by linearly independent vectors $v_1,\dots,v_N \in \R^N$. This result is given in Lemma~\ref{lattice-path-appendix} in the Appendix. We neither claim that it is new nor that it is optimal, but we could not find a reference for it and believe that it might be of independent interest. 

The paper is organized as follows. In Section~\ref{sec:main-results-general}, we present weak and strong maximum principles in the framework of bilinear forms.  In Section~\ref{notation} we collect useful properties and estimates related to the function spaces used in this paper. In Section~\ref{section-weak-mp}, we complete the proof of our weak maximum principles including Theorem~\ref{weak-max-intro}. In Section~\ref{section-strong-mp}, we give the proofs of our strong maximum principles including Theorem~\ref{main-theorem}. Finally, the Appendix is devoted to purely geometric properties of lattices in $\R^N$ which are used in the proof of Theorem~\ref{main-theorem}.\\ 

In the remainder of the paper, we will use the following notation. Let $U,V\subset \R^N$ be nonempty measurable sets, $x \in \R^N$ and $r>0$. We denote by $1_U: \R^N \to \R$ the characteristic function, $|U|$ the Lebesgue measure, and $\diam(U)$ the diameter of $U$. The notation $V \subset \subset U$ means that $\overline V$ is compact and contained in the interior of $U$. The distance between $V$ and $U$ is given by $\dist(V,U):= \inf\{|x-y|\::\: x \in V,\, y \in U\}$. Note that this notation does {\em not} stand for the usual Hausdorff distance.  If $V= \{x\}$ we simply write $\dist(x,U)$. We let $B_r(U):=\{x\in \R^N\;:\; \dist(x,U)<r\}$, so that $B_r(x):=B_r(\{x\})$ is the open ball centered at $x$ with radius $r$. We also put $B:=B_1(0)$ and $\omega_N:=|B|$. Finally, given a function $u: U\to \R$, $U \subset \R^N$, we let $u^+:= \max\{u,0\}$ and $u^-:=-\min\{u,0\}$ denote the positive and negative part of $u$, and we write $\supp\ u$ for the support of $u$ given as the closure in $\R^N$ of the set ${\{ x\in U\;:\; u(x)\neq 0\}}$.

\subsection*{Acknowledgement}

The authors thank Moritz Kassmann for valuable discussions. 

\section{Main results in a general setting}
\label{sec:main-results-general}

In this section, we consider a general setting of nonlocal equations extending the framework of Section~\ref{sec:introduction}. More precisely, we consider the equation
\begin{equation}
\label{general-1}
Iu= c(x) u +g \qquad \text{in $\Omega$}
\end{equation}
where $\Omega\subset \R^N$ is an open set, $c,g \in L^{\infty}(\Omega)$, and $I$ is a nonlocal linear operator formally given by 
\begin{equation}
  \label{eq:def-Iu-formal-0-J}
Iu(x)=P.V.\int_{\R^N}(u(x)-u(y))J(x,y)\ dy : = \lim_{\eps \to 0^+} 
\int_{\R \setminus B_{\eps}(x)}(u(x)-u(y))J(x,y)\ dy.
\end{equation}
Here and throughout the remainder of the paper, the measurable kernel function $J:\R^N\times \R^N  \to[0,\infty]$ is assumed to satisfy
\begin{enumerate}
	\item[(J1)] $J(x,y)=J(y,x)$ for all $x,y\in \R^N$, and
          \begin{equation}
\label{def-C-J}
        C_J:= \sup\limits_{x\in \R^N}\int\limits_{\R^N} 1 \wedge|x-y|^2 J(x,y)\ dy <\infty.    
          \end{equation}
\end{enumerate}
If $j: \R^N \to [0,\infty]$ satisfies assumption (j1) from the introduction, then the operator in (\ref{eq:def-Iu-formal}) arises as a special case of (\ref{eq:def-Iu-formal-0-J}) with the kernel $J:\R^N\times \R^N  \to[0,\infty],\; J(x,y)= j(x-y)$ which then satisfies (J1).  The bilinear form associated to $I$ is given by 
\begin{equation}\label{form}
(u,v) \mapsto \cE_J(u,v)=\frac{1}{2}\int_{\R^N}\int_{\R^N}(u(x)-u(y))(v(x)-v(y))J(x,y)\ dxdy.
\end{equation}
As we shall see in Lemma~\ref{Lipschitz-inclusion} below, assumption (J1) guarantees that $\cE_J$ is well-defined on the space of compactly supported Lipschitz functions. In particular, it is densely defined on $L^2(\Omega)$, where -- here and in the following -- we identify $L^2(\Omega)$ with the space of functions $u \in L^2(\R^N)$ with $u \equiv 0$ on $\R^N \setminus \Omega$.

We shall analyze supersolutions of \eqref{general-1} in weak sense (cf. \cite{DK15,FKV13}). For this we introduce the following function spaces. 
\begin{defi}
\label{space1}\label{space2}
For an open set $\Omega\subset \R^N $, we let 
$\cD^J(\Omega)$ denote the space of all functions $u \in L^2_{loc}(\R^N)$ with $\cE_J(u,u)<\infty$ and  $u \equiv 0$ 
on $\R^N \setminus \Omega$.

Moreover, we let 
$\cV^J(\Omega)$ denote the space of all functions $u\in L^2_{loc}(\R^N)$ such that 
$$
\rho(u,\Omega):= \frac{1}{2}\int_{\Omega}\int_{\R^N} (u(x)-u(y))^2J(x,y)\ dxdy<\infty. 
$$
\end{defi}

It is easy to see that $\cD^J(\Omega)$ and $\cV^J(\Omega)$ are indeed vector spaces. Moreover, $\cE_J$ defines a semi-definite scalar product on $\cD^J(\Omega)$. We also note the inclusions  
\begin{equation}
\label{eq:inclusions}
\cD^J(\Omega_1) \subset \cD^J(\Omega_2) \subset \cD^J(\R^N) = \cV^J(\R^N) \subset \cV^J(\Omega_2) \subset \cV^J(\Omega_1)
\end{equation}
for open subsets $\Omega_1 \subset \Omega_2 \subset \R^N$. Moreover, we shall see in Lemma~\ref{testing} below that $\cE_J(u,v)$ is well-defined for $u \in \cV^J(\Omega)$, $v \in \cD^J(\Omega)$. Thus we may define the following notions of sub- and supersolutions of (\ref{general-1}). 

\begin{defi}
\label{def-variational-supersolution}
We call $u\in \cV^J(\Omega)$ a \textit{variational supersolution} of (\ref{general-1}) if 
\begin{equation}\label{def-super}
\cE_J(u,v)\geq \int_{\Omega} [c(x)u(x)+g(x)] v(x)\ dx 
\end{equation}
for all $v\in \cD^J(\Omega)$, $v\geq 0$ with bounded support in $\R^N$. Similarly, we say that $u$ is a \textit{variational subsolution} of \eqref{general-1} if (\ref{def-super}) holds with a reversed inequality. \end{defi}
We point out that this notion of super- and subsolutions differs from the one in Section~\ref{sec:introduction} with regard to the test functions considered. Here we allow test functions $v \in \cD^J(\Omega)$ with bounded support in $\R^N$, whereas in Section~\ref{sec:introduction} the smaller space $\cC_c^\infty(\Omega)$ of standard test functions is considered. We used the terms {\em variational super- and subsolutions} because test functions $v \in \cD^J(\Omega)$ naturally appear in a variational formulation of~(\ref{general-1}).    
Note that without imposing further hypotheses, we cannot expect that $\cC_c^\infty(\Omega)$ is dense in $\cD^J(\Omega)$. A useful density result under additional assumptions will be derived below in Proposition~\ref{density}. 

To state a weak maximum principle for variational supersolutions of (\ref{general-1}), we need to consider 
\begin{equation}\label{firsteigen}
\Lambda_1(\Omega):=\inf_{u\in \cD^J(\Omega)}\frac{\cE_J(u,u)}{\|u\|_{L^2(\Omega)}^2}\in[0,\infty).
\end{equation}
Here the quotient is understood in the sense that $\frac{\cE_J(u,u)}{\|u\|_{L^2(\Omega)}^2} = 0$ if $\cE_J(u,u)= 0$ or $u \not \in L^2(\Omega)$. The value $\Lambda_1(\Omega)$ is the infimum of the spectrum of the self-adjoint realization of the operator $I$ in $L^2(\Omega)$, which is associated to the closed and densely defined symmetric bilinear form $\cE_J$ with domain $\cD^J(\Omega) \cap L^2(\Omega)$, see e.g. \cite[Theorem VIII.15, pp. 278]{SR}. Indeed, the closedness of $\cE_J$ on $\cD^J(\Omega) \cap L^2(\Omega)$ can be seen exactly as in \cite[Proposition 2.2]{JW14}, and it implies that $\cD^J(\Omega) \subset L^2(\Omega)$ is a Hilbert-space with scalar product $\cE_J$ if $\Lambda_1(\Omega)>0$. Since we will not use these functional analytic properties, we skip the details at this point. 

Consider the function 
\begin{equation}\label{smallj}
z \mapsto j(z):=\essinf \bigl \{J(x,x\pm z)\;:\;x\in \R^N \bigr \}
\end{equation}
which is an a.e. well-defined even and measurable function $j: \R^N \to [0,\infty]$ satisfying assumption (j\,\!1) from Section~\ref{sec:introduction} as a consequence of (J1). We note the following important observation proved in \cite[Lemma 2.7]{FKV13}: {\em If the set $\{j>0\} \subset \R^N$ has positive measure and $\Omega \subset \R^N$ is open and bounded, then $\Lambda_1(\Omega)>0$.} 

The following weak maximum principle indicates the significance of $\Lambda_1(\Omega)$.

\begin{prop}\label{wmp1}
Let (J1) be satisfied, let $\Omega\subset \R^N$ be open, bounded such that $\Lambda_1(\Omega)>0$, let $c\in L^{\infty}(\Omega)$ with $\|c^+\|_{L^{\infty}(\Omega)}<\Lambda_1(\Omega)$, and let $g\in L^2(\Omega)$. Then for every variational supersolution $u \in \cV^J(\Omega)$ of (\ref{general-1}) satisfying 
        \begin{equation}
\label{wmp1-cond-1}
u \ge 0 \qquad \text{a.e. in $\R^N \setminus \Omega$}          
        \end{equation}
we have 
\begin{equation}
  \label{wmp1-assertion}
\|u^-\|_{L^2(\Omega)}\leq \frac{\|g^{-}\|_{L^2(\Omega)}}{\Lambda_1(\Omega)-\|c^+\|_{L^{\infty}(\Omega)}}.  
\end{equation}
In particular, if $g \ge 0$ in $\Omega$, then  $u\geq0$ in $\R^N$.
\end{prop}

As an intermediate step in the proof of strong maximum principles, we need a variant of this weak maximum principle for domains with small volume. For this we need to analyze the small volume asymptotics of $\Lambda_1(\Omega)$. Let 
$$
\Lambda_{1}(r):= \inf\{\Lambda_1(\Omega)\::\: \text{$\Omega \subset \R^N$ open, $|\Omega|=r$}\} \qquad \text{for $r>0$.} 
$$
We shall see in Proposition~\ref{blowup-1} below that 
\begin{equation}
\label{blowup}\lim_{r \to 0}\Lambda_{1}(r) \ge \int_{\R^N}j(z)\,dz,
\end{equation}
where $j$ is given in (\ref{smallj}). A combination of Proposition~\ref{wmp1} and \eqref{blowup} readily yields the following small volume maximum principle.

	\begin{thm}\label{weak-max-small-volume}
Suppose that (J1) is satisfied, and let $c \in L^\infty(\R^N)$ satisfy 
\begin{equation}
  \label{eq:assumption-j3}
\|c^+\|_{L^\infty(\R^N)}< \int_{\R^N}j(z)\,dz,
\end{equation}
where $j$ is given in (\ref{smallj}). Then there exists $r>0$ such that for every open bounded set $\Omega \subset \R^N$ with $|\Omega| \le r$ and 
every variational supersolution $u \in \cV^J(\Omega)$ of 
\begin{equation}
\label{general-1-1}
Iu= c(x) u\qquad \text{in $\Omega$}
\end{equation}
with $u \ge 0$ a.e. in $\R^N \setminus \Omega$ we have $u \ge 0$ in $\R^N$.
\end{thm}

Our main result in this abstract setting is a strong maximum principle. For this we need the following further assumptions. 

\begin{enumerate}
	\item[(J2)] For every $r>0$, the function $j$ given in (\ref{smallj}) does not vanish identically on $B_r(0)$.
	\item[(J3)] We have
	$$
	\sup\limits_{x\in \R^N}\int\limits_{\R^N} 1 \wedge |z|\; |J(x,x+z)-J(x,x-z)|\ dz <\infty.
	$$
\end{enumerate}

\begin{thm}\label{hopf-simple2-variant1}
Assume (J1) -- (J3), let $\Omega\subset\R^{N}$ be a domain, and let $u \in \cV^J(\Omega)$ be a variational supersolution of 
\begin{equation*}
Iu= c(x) u\qquad \text{in $\Omega$}
\end{equation*}
satisfying $u \ge 0$ in $\R^N$. Then either $u \equiv 0$ in $\Omega$, or $u$ is strictly positive in $\Omega$.
	\end{thm}

Note that in this theorem we assume $\Omega$ to be a domain, i.e. a connected open set. A variant of this strong maximum principle for arbitrary open sets $\Omega \subset \R^N$ can be obtained if (J2) is replaced by the following much stronger uniform positivity condition. 
\begin{enumerate}
	\item[(J2)$_{S}$] For the function $j$ given in (\ref{smallj}) and every $r>0$, we have $\underset{B_r(0)}{\essinf}\: j >0$. 
\end{enumerate}

	\begin{thm}
\label{hopf-simple21}            
Assume (J1), (J2)$_{S}$, (J3),  and let $\Omega\subset\R^{N}$ be an open set. 
Furthermore, let $c \in L^\infty(\Omega)$, and let $u \in \cV^J(\Omega)$ be a variational supersolution of 
\begin{equation*}
Iu= c(x) u\qquad \text{in $\Omega$}
\end{equation*}
satisfying $u \ge 0$ in $\R^N$. Then either $u \equiv 0$ in $\R^N$, or $u$ is strictly positive in $\Omega$.
	\end{thm}

We emphasize that the alternative in Theorem~\ref{hopf-simple21} is stronger than the one in Theorem~\ref{hopf-simple2-variant1}. 
We also remark that, due to local uniform positivity of the kernel $J$ assumed in (J2)$_{S}$, the proof of Theorem~\ref{hopf-simple21} is much simpler than the proof of Theorem~\ref{hopf-simple2-variant1}.

\section{Preliminaries on the functional analytic setting}\label{notation}

In the following we keep using the notation from the previous section, and we assume (J1) throughout this section. The following statement ensures that our definition of variational supersolution of (\ref{general-1}) is well-defined.

\begin{lemma}\label{testing}
	Let $\Omega\subset \R^N$ be open. Then we have 
        \begin{equation}
          \label{eq:inequality-well-defined}
\int_{\R^N} \int_{\R^N} |u(x)-u(y)|\cdot|v(x)-v(y)|J(x,y)\ dxdy\le (2+ \sqrt{2}) \rho(u,\Omega)^{\frac{1}{2}} \cE_J(v,v)^{\frac{1}{2}}<\infty
        \end{equation}
for every $u \in \cV^J(\Omega)$, $v \in \cD^J(\Omega)$. Hence the bilinear form 
$\cE_J$ is well-defined and continuous on $\cV^J(\Omega)\times \cD^J(\Omega)$.
\end{lemma}

\begin{proof}
	Let $u\in \cV^J(\Omega)$, $v\in \cD^J(\Omega)$. Then, by the Cauchy-Schwarz inequality,
	\begin{align*}
	&\int_{\R^N} \int_{\R^N} |u(x)-u(y)|\cdot|v(x)-v(y)|J(x,y)\ dxdy
\\
&=  \int_{\Omega} \int_{\R^N} |u(x)-u(y)|\cdot|v(x)-v(y)|J(x,y)\ dxdy+\!\!\int_{\R^N\setminus \Omega}\,\int_{\R^N}|u(x)-u(y)|\cdot|v(x)-v(y)|J(x,y)\ dxdy\\
&\leq 2\rho(u,\Omega)^{\frac{1}{2}}\rho(v,\Omega)^{\frac{1}{2}} +\int_{\Omega}\int_{\R^N\setminus \Omega}|v(x)| |u(x)-u(y)|J(x,y)\ dydx\\
&\leq 2\rho(u,\Omega)^{\frac{1}{2}}\cE_J(v,v)^{\frac{1}{2}}  
+\Bigl(\int_{\Omega}\int_{\R^N \setminus \Omega }|v(x)|^2J(x,y)\ dydx\Bigr)^{\frac{1}{2}}
\Bigl(\int_{\Omega}\int_{\R^N \setminus \Omega }|u(x)-u(y)|^2J(x,y)\ dydx\Bigr)^{\frac{1}{2}}\\
&\leq 2\rho(u,\Omega)^{\frac{1}{2}}\cE_J(v,v)^{\frac{1}{2}}  
+\Bigl(\frac{1}{2}\int_{\R^N}\int_{\R^N}|v(x)-v(y)|^2J(x,y)\ dxdy\Bigr)^{\frac{1}{2}}
\Bigl(\int_{\Omega}\int_{\R^N}|u(x)-u(y)|^2J(x,y)\ dxdy\Bigr)^{\frac{1}{2}}\\
&\leq (2+ \sqrt{2})\rho(u,\Omega)^{\frac{1}{2}}\cE_J(v,v)^{\frac{1}{2}}<\infty.
	\end{align*}
	The continuity of the bilinear form $\cE_J$ on $\cV^J(\Omega)\times \cD^J(\Omega)$ now follows immediately from this bound.
\end{proof}

Next, we collect some elementary estimates for functions in $\cV^J(\Omega)$ and in $\cD^J(\Omega)$.

\begin{lemma}\label{pospart}
	Let $\Omega\subset \R^N$ open and $u\in \cV^J(\Omega)$. Then we have: 
        \begin{enumerate}
\item[(i)] If $u\equiv 0$ on $\R^{N}\setminus \Omega$, then $u\in \cD^{J}(\Omega)$.      
\item[(ii)] $u^{\pm}\in \cV^J(\Omega)$ and $\rho(u^\pm,\Omega)\leq \rho(u,\Omega)$.
\item[(iii)] If $u^+\in \cD^J(\Omega)$ or $u^-\in \cD^J(\Omega)$, then $\cE_J(u^+,u^-)$ is well-defined with $\cE_J(u^+,u^-) \le 0$.
\item[(iv)] If $u \ge 0$ on $\R^N \setminus \Omega$, then $u^- \in \cD^J(\Omega)$ and 
	\begin{equation}\label{eq:key-ineq1}
	\cE_J(u^-,u^-)\leq -\cE_J(u,u^-),
	\end{equation}
where the RHS is well-defined by Lemma~\ref{testing}. 
 \end{enumerate}
\end{lemma}

\begin{proof}
(i) By assumption, we have that 
        \begin{align*}
\cE_J(u,u)&=\rho(u,\Omega)+\int_{\R^N\setminus \Omega}\int_{\R^N}(u(x)-u(y))^2J(x,y)\ dxdy\\
&=\rho(u,\Omega)+\int_{\R^N\setminus \Omega}\int_{\Omega}(u(x)-u(y))^2J(x,y)\ dxdy \leq 2\rho(u,\Omega).
        \end{align*}
(ii) For $x,y \in \R^N$ we have 
\begin{equation}
  \label{eq:pointwise}
(u^+(x)-u^+(y))(u^-(x)-u^-(y))=-2\bigl(u^+(x)u^-(y)+u^-(x)u^+(y)\bigr) \le 0.
\end{equation}
Then 
\begin{align*}
\rho(u,\Omega)=& \rho(u^+ - u^-,\Omega)=\rho(u^+,\Omega)+\rho(u^-,\Omega)\\
&\quad -2 \int_{\Omega}\int_{\R^N}(u^+(x)-u^+(y))(u^-(x)-u^-(y))J(x,y)dx dy \\ 
\ge& \rho(u^+,\Omega)+\rho(u^-,\Omega).
\end{align*}
Consequently, $u^\pm \in \cV^J(\Omega)$ and $\rho(u^\pm,\Omega) \le \rho(u,\Omega)$.\\ 
(iii) Since $u^+\in \cD^J(\Omega)$ or $u^-\in \cD^J(\Omega)$ by assumption,
$\cE_J(u^+,u^-)$ is well-defined by Lemma~\ref{testing}. Moreover, by (\ref{eq:pointwise}) we have $\cE_J(u^+,u^-) \le 0$.\\
(iv) Since $u \ge 0$ on $\R^N \setminus \Omega$, the function $u^- \in \cV^J(\Omega)$ satisfies $u^- \equiv 0$ on $\R^N \setminus \Omega$. Consequently 
$u^- \in \cD^J(\Omega)$ by (i). Moreover, we have 
$$
\cE_J(u,u^-)= \cE_J(u^+-u^-,u^-) = \cE_J(u^+,u-)-\cE_J(u^-,u^-)
$$
where all terms are well-defined by Lemma~\ref{testing}. Furthermore, $\cE_J(u^+,u^-) \le 0$ by (iii). Hence (\ref{eq:key-ineq1}) follows.
\end{proof}

Next we give the proof of \eqref{blowup}, which we restate for the reader's convenience.

\begin{prop}
\label{blowup-1}
Let 
$$
\Lambda_{1}(r):= \inf\{\Lambda_1(\Omega)\::\: \text{$\Omega \subset \R^N$ open, $|\Omega|=r$}\}\qquad \text{for $r>0$.}
$$
Then we have 
\begin{equation}
\label{eq:limit-inequality}
\lim_{r \to 0}\Lambda_{1}(r) \ge \int_{R^N}j(z)\,dz.
\end{equation}
\end{prop}

	\begin{proof}
The proof is a refinement and generalization of the argument in \cite[Lemma 2.7]{JW14}. We first note that 
$$
\Lambda_{1}(r_1)  \ge \Lambda_{1}(r_2)\qquad \text{for $0<r_1<r_2$,}
$$ 
which can easily be deduced from the fact that 
$$
\cD^J(\Omega_1) \subset \cD^J(\Omega_2)\qquad \text{for $\Omega_1,\Omega_2 \subset \R^N$ with $\Omega_1 \subset \Omega_2$.}
$$
Hence the limit in (\ref{eq:limit-inequality}) exists. We also recall that $J(x,y)\geq j(x-y)$ for all $x,y\in \R^N$, $x\neq y$. We set 
		$$
		M_c:= \{ z \in \R^N \::\: j(z) \ge c\} \qquad \text{and}\qquad M^c:= \{ z \in \R^N \::\: j(z) < c\}
		$$
		for $c \in [0,\infty]$. Moreover, we consider the decreasing rearrangement of $j$ given by 
$$
d:(0,\infty) \to [0,\infty], \qquad d(r)= \sup \{c \ge 0 \::\: |M_c| \ge r \}
$$
 We first note that 
		\begin{equation}
		\label{eq:est-dec-rearr-3}
		|M_{d(r)}| \ge r \qquad \text{for every $r>0$} 
		\end{equation}
		Indeed, this is obvious if $d(r)=0$, since $M_0= \R^N $. If $d(r)>0$, we have $|M_c| \ge r$ for every $c < d(r)$ by definition, whereas $|M_c| < \infty$ for every $c>0$ as a 
		consequence of the fact that $j \in L^{1}(\R^{N}\setminus B_1(0))$ by (J1). Consequently, since $M_{d(r)}= \underset{c < d(r)}{\bigcap} M_c$, we have $|M_{d(r)}| =   \inf \limits_{c < d(r)} |M_c| \ge r.$ Next we claim that 
		\begin{equation}
		\label{eq:est-dec-rearr}
		\Lambda_{1}(r)  \ge \int_{M^{d(r)}} j(z)\,dz +  d(r) \Bigl(|M_{d(r)}|- r\Bigr) \qquad \text{for $r>0$.}
		\end{equation}
		Indeed, let $r>0$ and $\Omega \subset \R^N$ be measurable with $|\Omega|=r$. 
		For $u\in \cD^{J}(\Omega)$ we have
		\begin{align}
		\cE_J(u,u)&=\frac{1}{2}\int_{\R^{N}}\int_{\R^{N}}(u(x)-u(y))^2J(x,y)\ dxdy \nonumber \\
		&\geq\frac{1}{2}\int_{\Omega}\int_{\Omega}(u(x)-u(y))^2 j(x-y)\ dxdy+\int_{\Omega}u^2(x)\int_{\R^{N}\setminus \Omega} j(x-y)\ dy \ dx  \nonumber\\
		&\geq \inf_{x\in \Omega}\biggl(\;\int_{\;\R^{N}\setminus \Omega_x} j(z)\ dz\biggr)\|u\|^2_{L^{2}(\Omega)} \label{eq:est-dec-rearr-4}
		\end{align}
		with $\Omega_x:=x+\Omega$. Let $d:=d(r)$. We then have 
$$
|M_d \setminus \Omega_x| -|\Omega_x \setminus M_d| = |M_d|- |M_d \cap \Omega_x| -\Bigl(|\Omega_x|- |M_d \cap \Omega_x|\Bigr)= |M_d|-|\Omega_x|= |M_d|-r
$$
for every $x \in \Omega$ and thus 
		\begin{align*}
		\int_{\R^{N}\setminus \Omega_x} j(y)\ dy&=\int_{\R^N \setminus M_d} j(y)\ dy +\int_{M_d\setminus \Omega_x}j(y)\ dy- \int_{\Omega_{x}\setminus M_d}j(y)\ dy\\
		&\ge \int_{M^d} j(y)\ dy + d \Bigl(|M_d\setminus \Omega_x| - 
		|\Omega_{x}\setminus M_d|\Bigr) = \int_{M^d} j(y)\ dy +d(|M_d|-r). 
		\end{align*}
		Combining this with (\ref{eq:est-dec-rearr-4}), we obtain (\ref{eq:est-dec-rearr}), as  claimed. Now, by definition, the decreasing rearrangement of $j$ satisfies $d(r) \to d_0:= \underset{\R^N}{\esssup}\: j$ as $r \to 0$ and thus, by monotone convergence  
		$$
   \lim_{r\to 0} \int_{M^{\,d(r)}} j(y)\ dy  =    \lim_{r\to 0} \int_{\{j <d(r)\}} j(y)\ dy = \int_{\{j <d_0\}} j(y)\ dy
= \int_{M^{\,d_0}} j(y)\ dy.
$$
Together with (\ref{eq:est-dec-rearr}), this shows that 
$$
\lim_{r \to 0} \Lambda_{1}(r) \ge \liminf_{r\to 0} \Bigl(\int_{M^{\,d(r)}} j(y)\ dy  + d(r)(|M_{d(r)}|-r)\Bigr)
=\int_{M^{\,d_0}} j(y)\ dy + \liminf_{r\to 0} d(r)(|M_{d(r)}|-r)\Bigr)
$$
In the case where $d_0< \infty$, we have, by monotone convergence,
$$
\lim_{r\to 0} d(r)(|M_{d(r)}|-r) = d_0 |M_{d_0}| 
$$
and thus
$$
\lim_{r \to 0} \Lambda_{1}(r) \ge \int_{M^{\,d_0}} j(y)\ dy + d_0 |M_d| = \int_{\R^N}j(z)\,dz,
$$
as claimed. In the case where $d_0= \infty$ we have 
$$
\int_{M^{\,d_0}} j(y)\ dy = \int_{\R^N} j(z)\,dz 
$$
and thus also   
$$
\lim_{r \to 0} \Lambda_{1}(r) \ge  \int_{\R^N}j(z)\,dz.
$$
The proof is finished.
	\end{proof}

In the following, we note some useful inclusions of function spaces.

\begin{lemma}\label{cutoff}
	Let $\Omega\subset \R^N$ be a bounded open set, $\phi: \R^N \to \R$ be a Lipschitz function with $\phi \equiv 0$ in $\R^N \setminus \Omega$ and $u\in \cV^J(\Omega)$. Then $\phi u\in \cD^J(\Omega)$.
\end{lemma}
\begin{proof}
By assumption, there exists a constant $C_\phi>0$ with 
$$
|\phi(x)-\phi(y)|\le C_\phi  (1 \wedge |x-y|)
\quad \text{for $x,y \in \R^N$.}
$$
Consequently, 
	\begin{align*}
	\rho(\phi u,\Omega) &= \frac{1}{2}\int_{\Omega}\int_{\R^N} [u(x)\phi(x)-u(y)\phi(y)]^2J(x,y)\ dydx\\
        &= \frac{1}{2}\int_{\Omega}\int_{\R^N} \bigl[u(x)\bigl(\phi(x)-\phi(y)\bigr)+ \phi(y)\bigl(u(x)-u(y)\bigr)\bigr]^2J(x,y)\ dydx\\
	&\leq  \int_{\Omega}u(x)^2 \int_{\R^N}[\phi(x)-\phi(y)]^2J(x,y)\ dy dx+ \int_{\Omega}\int_{\R^N} \phi(y)^2[(u(x)-u(y)]^2J(x,y)\ dydx \\
	&\leq  C_\phi^2 \int_{\Omega}u(x)^2 \int_{\R^N} 1 \wedge |x-y|^2 J(x,y)\ dy dx + \|\phi\|_{L^\infty(\R^N)}^2 
\int_{\Omega}\int_{\R^N} [(u(x)-u(y)]^2J(x,y)\ dy dx \\
	&\leq C_\phi^2 C_J \|u\|_{L^2(\Omega)}^2 + 2 \|\phi\|_{L^\infty(\R^N)}^2 \rho(u,\Omega) < \infty
	\end{align*}
with $C_J$ given in (J1). Hence $\phi u\in \cV^J(\Omega)$. Since $\phi u \equiv 0$ on $\R^N \setminus \Omega$ by assumption, Lemma \ref{pospart}(i) now implies that $\phi u\in \cD^J(\Omega)$.
\end{proof} 
 
\begin{cor}\label{Lipschitz-inclusion}
Let $\Omega \subset \R^N$ be a bounded open subset. If $\phi: \R^N \to \R$ is a Lipschitz function with $\phi \equiv 0$ on $\R^N \setminus \Omega$, then $\phi \in \cD^J(\Omega)$.   
\end{cor}

\begin{proof}
This follows from Lemma~\ref{cutoff} applied to $u \equiv 1 \in \cV^J(\Omega)$.
\end{proof}
 
\begin{lemma}
\label{Lipschitz-V-J}  
Let $\Omega\subset \R^N$ be a an open set, and let $u \in L^\infty(\R^N)$ be locally Lipschitz in $\Omega$. Then $u \in V^J(\Omega')$ for all open sets $\Omega' \subset \subset \Omega$. 
\end{lemma}

\begin{proof}
Let $\Omega' \subset \subset \Omega$ be an open set. By assumption, there exists a constant $C = C(u,\Omega')>0$ with 
$$
|u(x)-u(y)|\le C  (1 \wedge |x-y|)
\quad \text{for $x \in \Omega',\: y\in \R^N$.}
$$
Consequently, by (J1) we have 
\[
2\rho(u,\Omega') = \int_{\Omega'}\int_{\R^N} [u(x)- u(y)]^2J(x,y)\ dydx \le  C^2 \int_{\Omega'} \int_{\R^N} 1 \wedge |x-y|^2 J(x,y)\ dydx\le C^2|\Omega'| C_J< \infty,
\]
so that $u \in V^J(\Omega')$.
\end{proof}

Next, we define the space of functions $u \in \cV^J(\Omega)$ such that $I u$ is well-defined as an element in the topological dual $L^1(\Omega)^* \cong L^\infty(\Omega)$ of $L^1(\Omega)$.

\begin{defi}	\label{DJinfty}
 Let $\Omega \subset \R^N$ be an open set. Denote by $\cV^J_{\infty}(\Omega)$ the space of all functions $u \in \cV^J(\Omega)$ such that there exists a constant $C=C(u)>0$ with 
	\[
	|\cE_J(u,\phi)| \le C \int_{\Omega} |\phi(x)|\,dx \qquad \text{for all $\phi \in \cD^J(\Omega)\cap L^1(\Omega)$.}
	\]
	Moreover, for $u\in \cV^J_{\infty}(\Omega)$ we put
$$
	\|u\|_{\cV^{J}_{\infty}(\Omega)}:=\sup \bigl\{|\cE_J(u,\varphi)| \::\: \varphi \in \cD^J(\Omega) \cap L^1(\Omega),\: \int_{\Omega}|\varphi(x)|\,dx = 1 \bigr\}
$$
\end{defi}

 We note that, for $\Omega\subset \R^N$ open, $\|\cdot\|_{\cV^{J}_{\infty}(\Omega)}$ defines a semi-norm on $\cV^J_{\infty}(\Omega)$. Moreover, 
 \begin{equation}
   \label{eq:basic-est-V-infty-L-1}
	\cE_J(u,\varphi) \leq \|u\|_{\cV^{J}_{\infty}(\Omega)} \|\varphi\|_{L^1(\Omega)}\qquad\text{for $u\in \cV^J_{\infty}(\Omega)$ and $\phi \in \cD^J(\Omega) \cap L^1(\Omega)$.}
    \end{equation}
We also recall that $\cD^J(\Omega) \cap L^1(\Omega)$ is dense in $L^1(\Omega)$ by Corollary~\ref{Lipschitz-inclusion}. Hence, for fixed $u \in \cV^J_{\infty}(\Omega)$, the map $\varphi \mapsto \cE_J(u,\varphi)$ extends to a continuous linear form in $L^1(\Omega)^* \cong L^\infty(\Omega)$. In the case where $u \in C^2_c(\R^N)$ and (J3) holds, we shall see in the following lemma that the associated element in $L^\infty(\Omega)$ is precisely given by $Iu$ in the sense of (\ref{eq:def-Iu-formal-0-J}). We also note that 
\begin{equation}
  \label{eq:inclusions-1}
 \cV^J_\infty(\R^N) \subset \cV^J_\infty(\Omega)\qquad \text{for every open subset $\Omega \subset \R^N$.}   
\end{equation}

\begin{lemma}\label{representation}
	Assume (J3). Then $C^{2}_c(\R^N)\subset \cV^J_{\infty}(\R^N)$, and we have 
	\begin{equation}\label{represent-equal}
	\cE_J(u,v)=\int_{\Omega} [Iu](x)v(x)\ dx\quad\text{ for all $u\in C^2_c(\R^N)$ and $v \in \cD^J(\R^N) \cap L^1(\R^N)$,} 
	\end{equation}
	where $Iu\in L^\infty(\R^N) \cap C(\R^N)$ is given by (\ref{eq:def-Iu-formal-0-J}) for $x \in \R^N$.
\end{lemma}

\begin{proof}
Fix $u\in C^2_c(\R^N)$. Then we may write 
	\[
	u(x+z)=u(x)+\nabla u(x)\cdot z + g(x,z) \qquad \text{for $x,z \in \R^N$}
	\]
with a function $g:\R^N\times \R^N\to \R$ satisfying 
$$
C:= \sup_{x,z \in \R^N}\frac{|g(x,y)|}{|z|^2}<\infty.
$$
Consequently, we also have  
$$
|2u(x)-u(x-z)-u(x+z)|\leq 2 C|z|^2 \qquad\text{for all $x,z\in \R^N$.} 
$$
For $0<   \eps < \eps' \le 1$ and $x \in \R^N$, we now write 
	\begin{align*}
	I_{\epsilon}u(x)&:=\int_{\R^N\setminus B_{\epsilon}(0)}(u(x)-u(x+z))J(x,x+z)\ dz\qquad \text{ and }\\
	K_{\epsilon',\epsilon}(x)&:=I_{\epsilon'}u(x)-I_{\epsilon}u(x)=\int_{B_{\epsilon'}(0)\setminus B_{\epsilon}(0)}(u(x)-u(x+z))J(x,x+z)\ dz.
	\end{align*}
It readily follows from assumption (J1) that $I_\eps u \in L^\infty(\R^N) \cap C(\R^N)$ for every $\eps>0$. We also have
	\begin{align*}
	K_{\epsilon',\epsilon}(x) &=\frac{1}{2} \int_{B_{\epsilon'}(0)\setminus B_{\epsilon}(0)}(u(x)-u(x+z))J(x,x+z)+(u(x)-u(x-z))J(x,x-z)\ dz \\
	&=\frac{1}{2}\int_{B_{\epsilon'}(0)\setminus B_{\epsilon}(0)}\Bigl(2u(x)-u(x+z)-u(x-z))J(x,x-z)\\
	&\qquad\qquad\qquad\qquad\qquad +(u(x)-u(x+z))(J(x,x+z)-J(x,x-z)\Bigr)\ dz
 \end{align*}
and thus 
\begin{align*}
|K_{\epsilon',\epsilon}(x)|&\leq C\int_{B_{\epsilon'}(0)\setminus B_{\epsilon}(0)}|z|^2J(x,x-z)\ dz+
\frac{\|\nabla u\|_{L^\infty}}{2}\int_{B_{\epsilon'}(0)\setminus B_{\epsilon}(0)}|z| |J(x,x+z)-J(x,x-z)| dz 
\end{align*}
It thus follows from assumptions (J1) and (J3) that 
	\[
	\lim_{\epsilon'\to0}\; \sup_{\epsilon\in(0,\epsilon')}|K_{\epsilon',\epsilon}(x)|=0 \qquad \text{uniformly in $x \in \R^N$.}
	\]
Hence the limit $Iu(x)=\lim\limits_{\epsilon\to0}I_{\epsilon} u(x)$ exists for all $x \in \R^N$, and $Iu\in L^\infty(\R^N) \cap C(\R^N)$. Now to see \eqref{represent-equal}, we consider $v\in \cD^J(\R^N) \cap L^1(\R^N)$ and apply Lebesgue's Theorem to get 
	\begin{align*}
	\cE_J(u,v)&=\frac{1}{2} \lim_{\epsilon\to0}\: \underset{|x-y| \ge \eps}{\int \int} (u(x)-u(y))(v(x)-v(y))J(x,y)\ dxdy\\
        &= \lim_{\epsilon\to 0} \int_{\R^N} v(x) \int_{\R^N \setminus B_\eps(x)} (u(x)-u(y)) J(x,y)\ dy dx\\     
	&= \int_{\R^N} v(x) \:\lim_{\eps \to 0}\, \int_{\R^N \setminus B_\eps(0)} (u(x)-u(y)) J(x,y)\ dzdx =\int_{\R^N} [Iu](x) v(x)\ dx,
	\end{align*}
	as claimed.
\end{proof}

\section{Proof of weak maximum principles}
\label{section-weak-mp}
This section is devoted to the proof of weak maximum principles. More precisely, we complete the proofs of Theorems~\ref{weak-max-intro},
Theorem~\ref{wmp1} and Theorem~\ref{weak-max-small-volume}. We begin with the 

\begin{proof}[Proof of Theorem~\ref{wmp1}(completed)]
Let $u \in \cV^J(\Omega)$ be a variational supersolution of (\ref{general-1}) satisfying (\ref{wmp1-cond-1}). Then $u^-\in \cD^J(\Omega)$ by Lemma~\ref{pospart}(iv), and
\begin{align*}
-\Lambda_1(\Omega)\|u^-\|_{L^2(\Omega)}^2&\geq -\cE_J(u^-,u^-)\geq \cE_J(u,u^-)\geq \int_{\Omega} \bigl(c(x) u u^- + g u^-\bigr)\,dx\\
&\ge -\int_{\Omega}c^+(x)[u^-]^2(x)\ dx-\int_{\Omega}g^-(x)u^-(x)\ dx\\
&\ge -\|c^+\|_{L^\infty(\Omega)}\|u^-\|_{L^2(\Omega)}^2 - 
\|u^-\|_{L^2(\Omega)}\|g^-\|_{L^2(\Omega)}.
\end{align*}
Since $\|c^+\|_{L^\infty(\Omega)} < \Lambda_1(\Omega)$ by assumption, we thus conclude that $\|u^-\|_{L^2(\Omega)}\leq \frac{\|g^{-}\|_{L^2(\Omega)}}{\Lambda_1(\Omega)-\|c^+\|_{L^{\infty}(\Omega)}}$, as claimed.
\end{proof}

\begin{proof}[Proof of Theorem~\ref{weak-max-small-volume}(completed)]
Combining assumption~(\ref{eq:assumption-j3}) with Lemma~\ref{blowup-1}, we find $r>0$ such that $\|c^+\|_{L^{\infty}(\R^N)}<\Lambda_{1}(\Omega)$ for every open set $\Omega \subset \R^N$ with $|\Omega|<r$. Fix $\Omega$ with this property and let $u \in \cV^J(\Omega)$ be a variational supersolution of (\ref{general-1-1}). Applying Theorem~\ref{wmp1} with $g \equiv 0$ then yields $u \ge 0$ in $\Omega$, so that $u \ge 0$ in $\R^N$ by assumption. The proof is finished.
\end{proof}

It remains to complete the proof of Theorem~\ref{weak-max-intro}. For this we need a density property, since the supersolution property assumed in Theorem~\ref{weak-max-intro} refers to testing with functions in $C^\infty_c(\Omega)$. In the following, for a function $u \in L^1_{loc}(\R^N)$, we let $u_\eps: \R^N \to \R$ denote the usual mollification of $u$ given by 
$$
u_{\epsilon}:=\rho_{\epsilon}\ast u=\int_{\R^N}\rho_{\epsilon}(z)u(\cdot-z)\ dz,
$$
where $\rho_0 \in C^\infty_c(\R^N)$ is a nonnegative radial function supported in $B_1(0)$ with $\int_{\R^N}\rho_0(x)\,dx = 1$ and  
$\rho_{\epsilon}(x)=\epsilon^{-N}\rho(\epsilon^{-1}x)$ for $x \in \R^N$, $\eps>0$.

\begin{prop}\label{density}
	Let $j: \R^N \to [0,\infty]$ be a function satisfying the assumption (j\,\!1) from the introduction, and put $J(x,y)=j(x-y)$ for $x,y\in \R^N$. Moreover, 
let $\Omega \subset \R^N$ be an open set, and let $u\in \cD^J(\Omega)$ be a function, which vanishes outside a compact subset of $\Omega$. Then we have $u_\eps \in \cD^J(\Omega)$ for $\eps>0$ sufficiently small and 
$$
\lim_{\eps \to 0} \cE_J(u-u_\eps,u-u_\eps)=0.
$$
\end{prop}

\begin{proof}
Let $K \subset \Omega$ be a compact subset such that $u \equiv 0$ on $\R^N \setminus K$. Then $u_{\epsilon} \in C^{\infty}_c(\R^N)$ is supported in $K_\eps:= \{x \in \R^N\::\: \dist(x,K) \le \eps\}
$ for $\eps>0$. In particular, we have 
\begin{equation}
\label{eq:Omega-Omega'}
\text{$u_\eps \equiv 0$ on $\R^N \setminus \Omega$ for $\eps>0$ sufficiently small.}  
\end{equation}
Since $J(x,y)=j(x-y)$ and $\rho_{\epsilon}(z)=\rho_{\epsilon}(-z)$ for $z \in \R^N$, we have, by H\"older's inequality and Jensen's inequality,
	\begin{align}
	2&\cE_J(u_{\epsilon},u_{\epsilon})\notag\\
	&=\int_{\R^N}\int_{\R^N} \int_{\R^N}\int_{\R^N}\rho_{\epsilon}(z)\rho_{\epsilon}(z') [u(x-z)-u(y-z)] [u(x-z')-u(y-z')]  j(x-y)\ dz'dzdxdy\notag \\
	&=\int_{\R^N}\int_{\R^N} \int_{\R^N}\int_{\R^N}\rho_{\epsilon}(z)\rho_{\epsilon}(z') [u(x+z'-z)-u(y+z'-z)][u(x)-u(y)]   j(x-y)\ dz'dzdxdy \notag\\
	&=\int_{\R^N}\int_{\R^N} \int_{\R^N}\int_{\R^N}\rho_{\epsilon}(z-z')\rho_{\epsilon}(z') [u(x+z)-u(y+z)][u(x)-u(y)]   j(x-y)\ dz'dzdxdy \notag\\
	&=\int_{\R^N}\int_{\R^N} [u(x)-u(y)]j(x-y)  \int_{\R^N} [\rho_{\epsilon}\ast \rho_{\epsilon}](z)   [u(x+z)-u(y+z)]\ dzdxdy  \notag\\
	&\leq \sqrt{2}\cE_J(u,u)^{\frac{1}{2}}\Bigg(\ \int_{\R^N}\int_{\R^N}\Bigg(\ \int_{\R^N} [\rho_{\epsilon}\ast \rho_{\epsilon}](z)   [u(x+z)-u(y+z)]\ dz\Bigg)^{2}j(x-y)dxdy\Bigg)^{\frac{1}{2}} \notag\\
	&\leq \sqrt{2}\cE_J(u,u)^{\frac{1}{2}}\Bigg(\ \int_{\R^N} [\rho_{\epsilon}\ast \rho_{\epsilon}](z)\int_{\R^N} \int_{\R^N}[u(x+z)-u(y+z)]^2 j(x-y)dxdydz\Bigg)^{\frac{1}{2}}\\
&=\sqrt{2} \cE_J(u,u)^{\frac{1}{2}}\Bigg(\ \int_{\R^N} [\rho_{\epsilon}\ast \rho_{\epsilon}](z)
 \int_{\R^N} \int_{\R^N}[u(x)-u(y)]^2 j(x-y)dxdydz \Bigg)^{\frac{1}{2}}\\
&= 2\cE_J(u,u) \quad\text{ for all $\epsilon>0$.}
\label{conv3a}
\end{align}
Here we also used that
	\begin{equation}\label{rho-eps2}
	\int_{\R^N}[\rho_{\epsilon}\ast \rho_{\epsilon}](z)\ dz=\left(\ \int_{\R^N}\rho_{\epsilon}(z)\ dz \right)^2=1 \quad\text{ for all $\epsilon>0$.}
	\end{equation}
We thus conclude that $u_\eps \in \cD^J(\R^N)$ for every $\eps>0$, and that $u_\eps \in \cD^J(\Omega)$ for $\eps>0$ sufficiently small by (\ref{eq:Omega-Omega'}). Next, we define $U,U_\eps \in L^2(\R^N\times \R^N)$ by 
$$
U(x,y)=[u(x)-u(y)]\sqrt{j(x-y)},\qquad U_\eps(x,y)=[u_\eps(x)-u_\eps(y)]\sqrt{j(x-y)}.
$$
Since 
$$
\cE_J(u-u_\eps,u-u_\eps)= \int_{\R^N} \int_{\R^N} [(u-u_\eps)(x)-(u-u_\eps)(y)]^2 j(x-y)\,dx dy = \|U-U_\eps\|_{L^2(\R^N \times \R^N)}^2 
$$
for $\eps >0$, it remains to prove that 
\begin{equation}
  \label{eq:full-convergence}
U_\eps \to U \quad \text{in $L^2(\R^N \times \R^N)$ as $\eps \to 0$}.
\end{equation}
For this we note that 
\begin{equation}
\label{aepointwise}
u_\eps \to u \quad \text{in $L^2(\R^N)$ and a.e. pointwise in $\R^N$ as $\eps \to 0$,}
\end{equation}
see e.g. \cite[Section 4.2.1]{EG92}. As a consequence, $U_\eps \to U$ a.e. pointwise in $\R^N \times \R^N$. By Fatou's Lemma and \eqref{conv3a}, we then find that 
\begin{align*}
\|U\|_{L^2(\R^N \times \R^N)}^2 \leq \liminf_{ \eps \to 0} \|U_\eps\|_{L^2(\R^N \times \R^N)}^2
\le \limsup_{\eps \to 0 } \|U_\eps\|_{L^2(\R^N \times \R^N)}^2&= \limsup_{\eps \to 0} \cE_J(u_{\epsilon},u_{\epsilon})\\
& \leq\cE_J(u,u) = \|U\|_{L^2(\R^N \times \R^N)}^2.
\end{align*}
This shows the convergence 
\begin{equation}
\label{eq:normconvergence}
\|U_\eps\|_{L^2(\R^N \times \R^N)} \to \|U\|_{L^2(\R^N \times \R^N)} \quad \text{as $\eps \to 0$.}
\end{equation}
Next we claim that 
\begin{equation}
\label{eq:weakconvergence}
U_\eps \rightharpoonup U \qquad \text{in $L^2(\R^N \times \R^N)$.}
\end{equation}
To see this, we let $V \in L^2(\R^N \times \R^N)$ be arbitrary, and we let  
$$
V_n(x,y)= \left \{ 
  \begin{aligned}
  &V(x,y) &&\qquad \text{if $|x|,|y| \le n$ and $|x-y| \ge \frac{1}{n}$;}\\
  &0 &&\qquad \text{elsewhere.}    
  \end{aligned}
\right.
$$
It then follows from Lebesgue's theorem that 
\begin{equation}
  \label{eq:V-convergence}
\text{$V_n \to V$ in $L^2(\R^N \times \R^N)$ as $n \to \infty$.}  
\end{equation}
Moreover, since $\sqrt{j(x-y)} \le 1 + j(x-y)$, we deduce from (j1) and the definition of $V_n$ that the functions 
$$
h_1^n,h_2^n: \R^N \to \R,\qquad h_1^n(x)= \int_{\R^N}\sqrt{j(x-y)} V_n(x,y)\,dy,\qquad h_2^n(y)= \int_{\R^N}\sqrt{j(x-y)} V_n(x,y)\,dx 
$$
are bounded with bounded support, so in particular $h_1^n,h_2^n \in L^2(\R^N)$ for $n \in \N$. By (\ref{aepointwise}), we thus conclude that 
\begin{align*}
\lim_{\eps \to 0}&\int_{\R^N} \int_{\R^N} U_\eps(x,y)V_n(x,y) \,dxdy =\lim_{\eps \to 0} \int_{\R^N} \int_{\R^N} [u_\eps(x)-u_\eps(y)]\sqrt{j(x-y)} V_n(x,y) \,dxdy\\
&= \lim_{\eps \to 0} \int_{\R^N}u_\eps(x) h_1^n(x)\,dx - \int_{\R^N}u_\eps(y) h_2^n(y)\,dy = \int_{\R^N}u(x) h_1^n(x)\,dx - \int_{\R^N}u(y) h_2^n(y)\,dy\\ 
&= \int_{\R^N} \int_{\R^N} [u(x)-u(y)]\sqrt{j(x-y)} V_n(x,y) \,dxdy= \int_{\R^N} \int_{\R^N} U(x,y)V_n(x,y) \,dxdy
\end{align*}
for every $n \in \N$, which implies that 
\begin{align*}
&\limsup_{\eps \to 0 } \Bigl|\int_{\R^N} \int_{\R^N} (U_\eps-U)(x,y)V(x,y) \,dxdy\Bigr|= \limsup_{\eps \to 0 }\Bigl| \int_{\R^N} \int_{\R^N} (U_\eps-U)(x,y)(V-V_n)(x,y) \,dxdy\Bigr| \\  
&\le \limsup_{\eps \to 0 } \|U_\eps-U\|_{L^2(\R^N \times \R^N)} \|V-V_n\|_{L^2(\R^N \times \R^N)}
\le 2 \|U\|_{L^2(\R^N \times \R^N)} \|V-V_n\|_{L^2(\R^N \times \R^N)} 
\end{align*}
for every $n \in \N$. Consequently, 
$$
\lim_{\eps \to 0} \int_{\R^N} \int_{\R^N} (U_\eps-U)(x,y)V(x,y) \,dxdy = 0 \qquad \text{as $\eps \to 0$}
$$
by (\ref{eq:V-convergence}). Hence (\ref{eq:weakconvergence}) holds. Since $L^2(\R^N \times \R^N)$ is uniformly convex, (\ref{eq:full-convergence}) now follows from (\ref{eq:normconvergence}) and (\ref{eq:weakconvergence}).
\end{proof}

\begin{cor}
\label{sec:supersolution-variational-supersolution}
	Let $j: \R^N \to [0,\infty]$ be a function satisfying the assumption (j\,\!1) from the introduction, and put $J(x,y)=j(x-y)$ for $x,y\in \R^N$. 

Moreover, let $\Omega \subset \R^N$ be an open set, let $c \in L^\infty_{loc}(\Omega)$, and let $u \in \cV^j_{loc}(\Omega)$ be a weak supersolution of (\ref{nonlocal-general-0}) in the sense of (\ref{eq:assumption-supersol-variant-1}). Then, for any open set 
$\Omega' \subset \subset \Omega$, the function $u \in \cV^j(\Omega')$ is a variational supersolution of the equation 
\begin{equation}
  \label{eq:equation-omega-prime}
I u = c(x)u \qquad \text{in $\Omega'$.}  
\end{equation}
\end{cor}

\begin{proof}
Choose an open set $\Omega''$ with  
$$
\Omega' \subset \subset \Omega'' \subset \subset \Omega
$$ 
Since $u \in \cV^j_{loc}(\Omega)$, we have $u \in \cV^J(\Omega'') \subset \cV^J(\Omega')$. Let $v \in \cD^J(\Omega')$, $v \ge 0$. By Proposition~\ref{density},
\begin{equation}
  \label{eq:density-statement}
\lim_{\eps \to 0} \cE_J(v-v_\eps,v-v_\eps) \to 0 \qquad \text{as $\eps \to 0$.} 
\end{equation}
Moreover, $v_\eps \in \cD^J(\Omega'')$ for $\eps>0$ sufficiently small, whereas also $v_\eps \to v$ in $L^2(\Omega'')$ as $\eps \to 0$.
It then follows from Lemma~\ref{testing} and (\ref{eq:density-statement}) that 
$$
|\cE_J(u,v-v_\eps)| \le \frac{2+ \sqrt{2}}{2}\, \rho(u,\Omega'')^{\frac{1}{2}}\, \cE_J(v-v_\eps,v-v_\eps)^{\frac{1}{2}} \to 0 \qquad \text{as $\eps \to 0$.}
$$
Consequently, by (\ref{eq:assumption-supersol-variant-1}),  
$$
\cE_J(u,v)= \lim_{\eps \to 0}\cE_J(u,v_\eps) \ge \lim_{\eps \to 0} \int_{\Omega}c(x) u  v_\eps\,dx =
\lim_{\eps \to 0} \int_{\Omega''}c(x) u v_\eps\,dx =\int_{\Omega''}c(x) u v\,dx = \int_{\Omega}c(x) u v\,dx.  
$$
Hence $u$ is a variational supersolution of $Iu = c(x)u$ in $\Omega'$, as claimed.
\end{proof}

\begin{proof}[Proof of Theorem \ref{weak-max-intro}(completed)]
Let $r>0$ be given by Theorem~\ref{weak-max-small-volume}. Moreover, let $\Omega \subset \R^N$ be an open set with $|\Omega|<r$, and let $u \in \cV^j_{loc}(\Omega)$ be a weak supersolution of (\ref{nonlocal-general-0}) in $\Omega$ satisfying (\ref{eq:weak-max-intro-cond}). Then we may choose an open set $\Omega' \subset \subset \Omega$ such that $u \ge 0$ on $\R^N \setminus \Omega'$. We note that $|\Omega'| \le |\Omega| < r$. Moreover, $u$ is a variational supersolution of $Iu =c(x)u$ in $\Omega'$ by Corollary~\ref{sec:supersolution-variational-supersolution}. Consequently, Proposition~\ref{wmp1} -- applied with $\Omega'$ in place of $\Omega$ -- yields that $u \ge 0$ in $\R^N$. The proof is finished. 
\end{proof}

\section{Proof of strong maximum principles}
\label{section-strong-mp}
This section is devoted to the proof of strong maximum principles. More precisely, we will complete the proofs of Theorem~\ref{main-theorem}, Theorem~\ref{hopf-simple2-variant1}, and Theorem~\ref{hopf-simple21}. As a preparation, we first need to investigate the validity of the following key property of a pair of subsets $M,K \subset \R^N$. Throughout this section, we assume that the kernel function $J$ satisfies assumptions (J1)--(J3).

\begin{defi}
\label{SSP-property}
Let $K \subset \R^N$ be measurable with $\diam\ K < \infty$, and let $M \subset \R^N$ be open with $\dist(M,K)>0$. We say that the pair $(K,M)$ satisfies the {\em subsolution property -- (SSP) in short --} if for every $f \in \cV^J_\infty(M)$ and $\kappa>0$ there exists $a>0$ such that the function \mbox{$w:=f+a1_K\in \cV^J_\infty(M)$} is a variational subsolution of the equation 
$$
I w = -\kappa \qquad \text{in $M$.}
$$
\end{defi}
We recall here that indeed $w \in \cV^J_\infty(M) \subset\cV^J(M)$ since $f \in \cV^J_{\infty}(M)$ and $1_K \in \cV^J_\infty(M)$ as a consequence of the fact that $\dist(M,K)>0$ and $\diam\ K < \infty$.

\begin{lemma}
\label{SSP-basic}
Let $K \subset \R^N$ be measurable with $\diam\ K < \infty$ and let $M \subset \R^N$ be open with $\dist(M,K)>0$. 
If 
\begin{equation}
\label{pos-j}
\inf_{x\in M} \int_K J(x,y)\ dy>0,  
\end{equation}
then the pair $(M,K)$ satisfies (SSP).  
\end{lemma}

\begin{proof}
Let $f \in \cV^J_\infty(M)$ and $\kappa>0$. By \eqref{pos-j} we may choose $a>0$ sufficiently large such that 
		$$
		\|f\|_{\cV^{J}_{\infty}(U)}  - a \inf_{x \in M} \int_K J(x,y)\ dy \le  -\kappa.
		$$
We then put $w=f+a1_K \in \cV^J_\infty(M)$, and we note that for a nonnegative function $\varphi\in \cD^{J}(M)$ with bounded support we have 
		\begin{align*}
		\cE_J(w,\varphi)&=\cE_J(f,\varphi) +a\cE_J(1_K, \phi)\leq \|f\|_{\cV^{J}_{\infty}(M)} \int_{M}\varphi(x)\ dx -a\int_{M}\varphi(x)\int_{K}J(x,y)\ dydx \\
		&\leq \Bigl(\|f\|_{\cV^{J}_{\infty}(M)}-a\inf_{x\in M}\int_{K}J(x,y)\ dy\Bigr) \int_{M}\varphi(x)\ dx \le -\kappa \int_{M} \phi(x)\,dx.
\end{align*}
By definition, this implies that $w$ is a variational subsolution of the equation $I w = -\kappa$ in $M$, as claimed.
	\end{proof}

	\begin{lemma}\label{positivity-inherit}        
Let $\Omega\subset\R^{N}$ be an open set, let $c \in L^\infty(\Omega)$ and suppose that
 \begin{equation}
   \label{eq:assumption-supersol-variant-1-0}
\left\{
  \begin{aligned}
&\text{$u \in \cV^J(\Omega)$ is a variational supersolution of the}\\
&\text{equation $Iu = c(x)u\:$ in $\Omega$ with $u\geq 0$ in $\R^{N}$.}
  \end{aligned}
\right.
\end{equation}
Furthermore, let $K \subset \R^N$ be measurable with 
$$
|K|>0, \quad \diam\ K < \infty \quad \text{and}\quad  \underset{K}\essinf\: u >0,
$$
and let $M \subset \Omega$ be open with $\dist(M,K)>0$ and such that the pair $(K,M)$ satisfies (SSP).\\[0.1cm]
Then $u$ is strictly positive in $M$.
	\end{lemma}

        \begin{proof}
Since 
$$
\cE_J(u,\phi) \ge \int_{\Omega} c(x) u \phi\,dx \ge - \int_{\Omega} c^-(x) u\phi\,dx 
$$
for every $\phi \in \cD^J(\Omega)$, $\phi \ge 0$ with bounded support by (\ref{eq:assumption-supersol-variant-1-0}), we may assume that $c \le 0$ in the following. It suffices to show that for every $x_0 \in M$ there exists $r>0$ such that $\underset{B_{r}(x_0)}\essinf\: u >0$. So let $x_0 \in M$. Since $\|c^+\|_{L^\infty(\Omega)}=0 < \int_{\R^N}j(z)\,dz$ by (J2), we may -- extending $c$ trivially to all of $\R^N$ -- apply Theorem~\ref{weak-max-small-volume} to find $r>0$ sufficiently small such that $B_{2r}(x_0) \subset M$ and such that 
 \begin{equation}
   \label{eq:assumption-supersol-variant-1-0-1}
\left\{
  \begin{aligned}
&\text{every variational supersolution $v \in \cV^J(B_{2r}(x_0))$ of the equation $I v = c(x) v$}\\
&\text{in $B_{2r}(x_0)$ with $v \ge 0$ in $\R^{N} \setminus B_{2r}(x_0)$ satisfies $v \ge 0$ in $B_{2r}(x_0)$.} 
  \end{aligned}
\right.
\end{equation}
We then pick a function $f\in C^{2}_{c}(\R^N)$ such that $0 \le f \le 1$ and 
		\begin{equation*}
		f(x):=\left\{\begin{aligned} 
                &1 &&\qquad \text{ for $|x-x_0|\leq r$,}\\
		&0 &&\qquad \text{ for $|x-x_0|\geq 2r$.}\\
		\end{aligned}\right.
		\end{equation*}
                We note that $f \in \cV^J_\infty(M)$ by (\ref{eq:inclusions-1}) and 
Lemma~\ref{representation}. Moreover, since the pair $(K,M)$ satisfies property (SSP), there exists $a>0$ such that the 
		function $w := f + a 1_K \in \cV^{J}_\infty(M)$ is a variational subsolution of the equation  
$$
I w = - \|c\|_{L^\infty(M)} \qquad \text{in $M$.}
$$
Since $0 \le w \le 1$ in $B_{2r}(x_0)$, this clearly implies that 
		\begin{equation}
		\label{eq:w-ineq}
		 \cE_J(w,\varphi)  \le -\|c\|_{L^\infty(M)} \int_{B_{2r}(x_0)}\varphi(x)\,dx \leq \int_{B_{2r}(x_0)}c(x)w(x)\varphi(x)\ dx
		\end{equation}
for $\varphi \in \cD^J(B_{2r}(x_0)),\varphi \ge 0$. Note that the function $w$ also satisfies  
		\begin{equation}
		\label{eq:ineq-w}
		w\equiv 0 \quad \text{on $\R^{N} \setminus (B_{2r}(x_0) \cup K)$,} \quad \qquad w \equiv  a 
		\quad \text{on $K$.}
		\end{equation}
		Let $\delta:= \underset{K}\essinf\: u$, so that $\delta>0$ by assumption. By (\ref{eq:assumption-supersol-variant-1-0}), (\ref{eq:w-ineq}), and (\ref{eq:ineq-w}) the function 
$$
v:=u-\frac{\delta}{a} w \:\in\: \cV^{J}(B_{2r}(x_0))
$$ 
is a variational supersolution of the equation $I v = c(x) v$ in $B_{2r}(x_0)$ with $v \ge 0$ on $\R^{N} \setminus B_{2r}(x_0)$. 
By (\ref{eq:assumption-supersol-variant-1-0-1}), we thus conclude that $v \ge 0$ in $\R^N$, so that 
$$
u \geq \frac{\delta}{a} w = \frac{\delta}{a}  >0 \qquad \text{in $B_{r}(x_0)$,}
$$
as desired.
        \end{proof}

We may now complete the 
	
	\begin{proof}[Proof of Theorem~\ref{hopf-simple21}]
 We assume that $u \not \equiv 0$  in $\R^N$. For given $x_0 \in \Omega$, it then suffices to show that $\underset{B_r(x_0)}\essinf\: u >0$ for $r>0$ sufficiently small. Since $u\not\equiv 0$ in $\R^N$ there exists a bounded measurable set 
		$K \subset \R^{N}$ with $|K|>0$, $\diam\ K < \infty$, $x_0 \not \in \overline K$ and such that $\underset{K}\essinf\: u >0$.
We may then fix $0<r< \frac{1}{2}\dist(x_0,K \cup \partial \Omega)$ and set $M:= B_{2r}(x_0) \subset \Omega$. Then 
$$
\inf_{x\in M} \int_K J(x,y)\ dy \ge \inf_{x\in M} \int_K j(x-y)\ dy \ge |K| \:\underset{B_R(0)}{\essinf}\: j >0 
$$
for $R>0$ sufficiently large as a consequence of assumption (J2)$_{S}$, and thus (\ref{pos-j}) is satisfied. Consequently, the pair $(K,M)$ satisfies (SSP) by Lemma~\ref{SSP-basic}, and thus Lemma~\ref{positivity-inherit} implies that $u$ is strictly positive in $M$. The proof is thus finished.
	\end{proof}
	
The next task of this section is to complete the proof of Theorem~\ref{hopf-simple2-variant1}, in which the global positivity assumption (J2)$_{S}$ for the kernel is not assumed. We recall that, due to the fact that we are not assuming the global positivity condition (J2)$_{S}$, we cannot expect that $u \not \equiv 0$ in $\R^N$ implies $u \not \equiv 0$ in $\Omega$. Hence the alternative in Theorem~\ref{hopf-simple2-variant1} differs from the alternative in Theorem~\ref{hopf-simple21}.  
We need to recall a measure theoretic notion. Let $K \subset \R^N$ be a measurable subset. A point $x \in \R^N$ is called a \textit{point of density one for $K$} if 
\begin{equation}
  \label{eq:density-one}
\lim_{r \to 0} \frac{|B_r(x) \cap K|}{|B_r|}=1.
\end{equation}
By a classical result, 
\begin{equation}
  \label{eq:density-one-assertion}
\text{a.e. $x \in K$ is a point of density one for $K$,}
\end{equation}
see e.g. \cite[Corollay 3 in Section 1.7]{EG92}. We need the following result on the existence of pairs $(K,M)$ satisfying the subsolution property (SSP).

\begin{lemma}\label{subsolutionstart}
Let $K \subset \R^N$ be measurable with $|K|>0$, $\diam\ K<\infty$ and such that
every $x \in K$ is a point of density one for $K$.\\
Then for every $\epsilon>0$ there exists a nonempty open set $M \subset \subset B_{\epsilon}(K)\setminus \overline K$ such that the pair $(K,M)$ satisfies property (SSP).
\end{lemma}

\begin{proof}
Let $K$ be as assumed and let $\epsilon>0$. Assumption (J2) implies that there exists $\delta>0$ such that the set 
$$
A_0 := \{z \in B_{\eps}(0) \setminus \{0\}\::\: j(z) \ge \delta\}
$$
has positive measure. Using (\ref{eq:density-one-assertion}), we may then choose a subset $A \subset A_0$ with $|A|>0$ and such that every point in $A$ is of density one for $A$. Since $j$ is even and therefore $A_0=-A_0$, we may also assume that $A=-A$. We then fix $p\in A$. Since $\overline K$ is compact by assumption, there exists a point $v_0 \in \overline K$ which maximizes the function 
$$
\overline K \to \R,\qquad v \mapsto v \cdot p,
$$
where $\cdot$ denotes the euclidean inner product on $\R^N$. This readily implies that $\dist(v_0+p,K)=|p|$. 
Moreover, since $0<|p|<\eps$, we may choose $w_0 \in K$ with $|v_0-w_0|<\min \{
\frac{|p|}{2},\eps-|p|\}$. Putting $x_0:= w_0 +p$, we then have
$$
\dist(x_0,K) \ge \dist(v_0+p,K) -|v_0-w_0| \ge \frac{|p|}{2}>0
$$
and
$$
\dist(x_0,K) \le \dist(v_0+p,K) +|v_0-w_0| <\eps.
$$
Moreover, since $p$ is a point of density one for $A$ and $w_0$ is a point of density one for $K$ by assumption, we may fix $r>0$ such that 
$$
|B_r(p)\cap A| \ge \frac{3}{4} |B_r(0)| \qquad \text{and}\qquad  |B_r(w_0) \cap K| \ge \frac{3}{4} |B_r(0)|.
$$
Then we have
\begin{align*}
\int_K j(x_0-y)\ dy &= \int_{x_0-K}j(y)\,dy \geq \delta |(x_0-K) \cap A| \ge \delta |(x_0-K) \cap B_r(p) \cap A|\\
&\ge \delta \Bigl(|B_r(p) \cap A|- |B_r(p) \setminus (x_0-K)|\Bigr)= \delta \Bigl(|B_r(p) \cap A|- |B_r(w_0) \setminus K|\Bigr)\\
&\ge \delta \Bigl(\frac{3}{4}|B_r(0)|-\frac{1}{4}|B_r(0)|\Bigr)= \frac{\delta}{2}|B_r(0)|>0.
\end{align*}
Since the map $x\mapsto \int_K j(x-y)\ dy$ is continuous on $\R^N \setminus \overline K$, we find an open neighborhood $M\subset\subset B_{\epsilon}(K)\setminus \overline K $ of $x_0$ such that 
\begin{equation*}
\inf_{x\in M} \int_K j(x-y)\ dy>0.  
\end{equation*}
Thus (\ref{pos-j}) holds, and hence the pair $(K,M)$ satisfies (SSP) by Lemma~\ref{SSP-basic}.
 \end{proof}
	
The next ingredient in the proof of Theorem \ref{hopf-simple2-variant1} are certain properties of lattices. In the following, we let $v_1,\ldots, v_N \in \R^N$ be linearly independent, and we let 
$$
G:=\sum \limits_{k=1}^{N}\Z v_k \;\subset\; \R^N
$$ 
be the lattice generated by $v_1,\dots,v_N$. The following basic observation is well-known. For the readers convenience, we include the short proof in the appendix.

\begin{lemma}\label{lattice-intersection}
Let $r>\frac{1}{2}\sum \limits_{k=1}^N|v_k|$ and $x_0 \in \R^N$. Then $B_r(x_0) \cap G\neq \emptyset$. 
\end{lemma}

The next lemma is concerned with discrete paths in $G$ joining two given lattice points $x, x' \in G$. More precisely, we need an estimate for the diameter of those paths depending on the distance of $x$ and $x'$.

\begin{lemma}\label{lattice-path}	
For every $\rho \ge \sum \limits_{k=1}^{N}|v_k|$ and every $x,x' \in G$ with $|x-x'|<\rho$ there exist $n \in \N$ and points 
$$
w_0,\ldots, w_n \in B_{4^{N-1} \rho}(x')
$$
such that $w_0=x'$, $w_n=x$ and 
$$
w_{\ell}-w_{\ell-1} \in \{ \pm v_1,\dots, \pm v_N\} \qquad \text{for $\ell = 1,\dots,n$.}
$$
\end{lemma}

The proof of this lemma is somewhat complicated, and we could not find a result of this type in the literature. We do not claim that the factor $4^{N-1}$ is optimal, but it is sufficient for our purposes. We postpone the proof of Lemma~\ref{lattice-path} to the appendix. We are now prepared to complete the proof of Theorem~\ref{hopf-simple2-variant1}.

	\begin{proof}[Proof of Theorem~\ref{hopf-simple2-variant1} (completed)] 
		We suppose that $u\not\equiv 0$ in $\Omega$, and we let 
$$
W:= \Bigl \{y \in \Omega\::\: \text{$\underset{B_{r}(y)}\essinf\: u >0$ for $r>0$ sufficiently small} \Bigr\}
$$
We then need to show that $W=\Omega$. The proof is divided into three steps.\\[0.1cm] 
\textit{\underline{Claim 1:}} $W$ is nonempty.\\[0.1cm]
		To see this, we first note that, since $u \not \equiv 0$ in $\Omega$, there exists $\delta>0$ and 
a measurable subset $K \subset \Omega$ with $|K|>0$ and such that
\begin{equation}
  \label{eq:K-inf-1-var-2}
\underset{K}\essinf\: u >0.
\end{equation}
Making $K$ smaller if necessary, we may also assume that $\overline K \subset \Omega$ and $\diam\ K<\infty$. 
Moreover, using~(\ref{eq:density-one-assertion}) and removing a set of measure zero from $K$ if necessary, we may assume that every point $x \in K$ is a point of density one for $K$. Consequently, with $\eps:= \dist(K,\partial \Omega)>0$, Lemma~\ref{subsolutionstart} yields the existence of a nonempty open set 
$$
M \subset \subset B_{\epsilon}(K)\setminus \overline K \subset \Omega
$$ 
such that the pair $(K,M)$ satisfies property (SSP). Then (\ref{eq:K-inf-1-var-2}) and Lemma~\ref{positivity-inherit} yield that $u$ is strictly positive in $M$. Consequently, $M \subset W$ and $W$ is nonempty.
 Thus Claim 1 is proved.\\[0.1cm]
\textit{\underline{Claim 2:}} If $x \in W$, then also $B_{r(x)}(x) \subset W$, where 
\begin{equation}
  \label{eq:def-r-x}
r(x):= \frac{4^{-N}}{3}\dist(x,\partial \Omega) \:>\:0.
\end{equation}
To see this, let $x_0 \in W$, put $r_0:= r(x_0)$, and let $x_1 \in B_{r_0}(x_0)$. Moreover, let $\epsilon \in (0,r_0)$ such that 
\begin{equation}
  \label{eq:epsilon-essinf}
\underset{B_{\epsilon}(x_0)}\essinf\: u>0.  
\end{equation}
By (J2), we may choose $\eps_1  \in (0,\frac{\epsilon}{4N})$ such that the function $j$ does not vanish a.e. in $B_{2\eps_1}(0) \setminus B_{\eps_1}(0)$. Consequently, there exists a subset $A\subset B_{2\eps_1}(0) \setminus B_{\eps_1}(0)$ with $|A|>0$ and  
$$
\underset{A}\essinf\: j >0.
$$
Since $j$ is an even function, we may also assume that $A=-A$. Moreover, using (\ref{eq:density-one-assertion}) and removing a set of measure zero from $A$ if necessary, we may also assume that every point in $A$ has density one with respect to $A$.
 As a consequence, for $0<\eps'<\eps'' <\eps_1$ we have 
\begin{equation}
  \label{eq:positivity-integral-v-j-0}
\inf_{x \in B_{\eps'}(A)}\int_{B_{\eps''}(x)}j(z)\,dz >0.
\end{equation}
Since $A$ has positive measure, we may now choose linearly independent vectors $v_1,\ldots,v_{N}\in A$. By (\ref{eq:positivity-integral-v-j-0}), we then have 
\begin{equation}
  \label{eq:positivity-integral-v-j}
\inf_{x \in B_{\eps'}(\pm v_j)}\int_{B_{\eps''}(x)}j(z)\,dz >0 \qquad \text{for $j=1,\dots,N$ and $0<\eps'<\eps'' <\eps_1$.}
\end{equation}
Let $G:= \sum \limits_{j=1}^n \Z v_j \subset \R^N$ be the lattice generated by $v_1,\dots,v_N$. Since 
$$
\sum\limits_{j=1}^{N}|v_j|\leq 2N\epsilon_1<\frac{\epsilon}{2},
$$
Lemma~\ref{lattice-intersection} implies that the intersection of $B_{\frac{\epsilon}{2}}(x_0)$ with the translated lattice $x_1 +G$ is nonempty. Hence there exists 
\begin{equation}
  \label{eq:def-x-0-prime}
x_0'\in (x_1+G) \cap B_{\frac{\epsilon}{2}}(x_0).  
\end{equation}
Furthermore we have 
		$$
		|x_1-x_0'|\leq |x_1-x_0|+|x_0-x_0'|\leq r_0 + \frac{\eps}{2} < 2 r_0 
\quad\text{and}\quad \sum\limits_{j=1}^{N}|v_j|< \frac{\epsilon}{2} < 2 r_0 
		$$
Hence Lemma \ref{lattice-path} yields the existence of 
\begin{equation}
  \label{eq:w_j-inclusion}
w_1,\ldots,w_n \in B_{2 \cdot 4^{N} r_0}(x_0') 
\end{equation}
such that $w_1=x_0'$, $w_n=x_1$, and 
\begin{equation}
  \label{eq:path-v-j-w-j}
w_{j+1}-w_j\in  \{\pm v_1,\dots,\pm v_N\} \qquad \text{for $j=1,\dots,n-1$.}
\end{equation}
For $j \in \N$ we now define $\epsilon_j:=\frac{\epsilon_1}{j}$, and we set 
$$
M_{j}:=B_{\epsilon_{2j-1}}(w_j)\quad \text{and}\quad  K_j:=B_{\epsilon_{2j}}(w_{j}) \subset M_j  \qquad \text{for $j=1,\dots,n$.}
$$
We note that 
\begin{equation}
  \label{eq:M-j-subset-Omega}
M_j \subset \Omega \qquad \text{for $j=1,\dots,N$,}
\end{equation}
since for $j=1,\dots,N$ we have
$$
\dist(w_j,\partial \Omega) \ge \dist(x_0',\partial \Omega)- 2 \cdot 4^{N} r_0 \ge \dist(x_0,\partial \Omega) - \frac{\eps}{2} - 2 \cdot 4^{N} r_0 = 4^{N} r_0 - \frac{\eps}{2} 
$$
by (\ref{eq:def-r-x}), (\ref{eq:def-x-0-prime}) and (\ref{eq:w_j-inclusion}) and thus 
$$
\dist(x,\partial \Omega) \ge \dist(w_j,\partial \Omega)-\eps_{2j-1} \ge 4^{N} r_0 - \frac{\eps}{2} -\eps_{1} \ge 4^{N} r_0 - \eps \ge (4^N-1)r_0>0 \quad 
\text{for $x \in M_j$.}
$$
We claim that 
\begin{equation}
  \label{eq:ssp-k-j-m-j+1}
\text{for $j=1,\dots,n-1$, the pair $(K_j,M_{j+1})$ satisfies property (SSP).}  
\end{equation}
Indeed, we clearly have $\diam\ K_j < \infty$, whereas $M_{j+1} \subset \R^N$ is an open set with 
$\dist(M_{j+1},K_j)>0$ since, by (\ref{eq:path-v-j-w-j}),
$$
|w_{j+1}-w_k| \ge \min_{j=1,\dots,N} |v_j| \ge \eps_1 > \eps_{2j} +\eps_{2j+1} \qquad \text{for $j=1,\dots,n-1$.} 
$$
Moreover, we have that 
\begin{align*}
\inf_{x \in M_{j+1}} \int_{K_j}J(x,y)\ dy \ge \inf_{x 
\in B_{\epsilon_{2j+1}}(w_{j+1})}\; \int_{B_{\eps_{2j}(w_j)}}j(x-y)\ dy&= \inf_{x \in B_{\epsilon_{2j+1}}(w_{j+1}-w_j)}\; \int_{B_{\eps_{2j}(w_j)}}j(x+w_j-y)\ dz\\
&= \inf_{x \in B_{\epsilon_{2j+1}}(w_{j+1}-w_j)}\; \int_{B_{\eps_{2j}(x)}}j(z)\ dz>0
\end{align*}
by (\ref{eq:positivity-integral-v-j}) and (\ref{eq:path-v-j-w-j}). Hence (\ref{eq:ssp-k-j-m-j+1}) follows from Lemma~\ref{SSP-basic}.
 Inductively, we now show that 
\begin{equation}
  \label{eq:induction-statement}
\underset{K_j}{\essinf}\:u >0 \qquad \text{for $j=1,\ldots,n$.}
\end{equation}
For $j=1$ this is true since, by definition and (\ref{eq:def-x-0-prime}), 
$$
K_1=B_{\frac{\eps_1}{2}}(w_{1})= B_{\frac{\eps_1}{2}}(x_0') \subset B_\eps(x_0)
$$ 
and (\ref{eq:epsilon-essinf}) holds. Moreover, if (\ref{eq:induction-statement}) holds for some $j \in \{1,\dots,n-1\}$, then (\ref{eq:M-j-subset-Omega}), (\ref{eq:ssp-k-j-m-j+1}) and Lemma~\ref{positivity-inherit} imply that $u$ is strictly positive on $M_{j+1}$, and therefore (\ref{eq:induction-statement}) holds for $j+1$ in place of $j$ since $K_{j+1} \subset \subset M_{j+1}$.\\
Applying (\ref{eq:induction-statement}) with $j=n$ yields 
$$
\underset{B_{\eps_{2n}}(x_1)}{\essinf}\:u =\underset{K_n}{\essinf}\:u \: >0
$$
and thus $x_1 \in W$. Hence Claim 2 is proved.\\[0.1cm]
We may now complete the proof of the theorem as follows. By Claim 1 we know that $W$ is nonempty. Moreover, by definition, $W \subset \Omega$ is open. Since $\Omega$ is connected, it thus suffices to show that $W$ is relatively closed in $\Omega$. To see this, let $(x_n)_n$ be a sequence in $W$ with 
$x_n \to x \in \Omega$ as $n \to \infty$. Since, by definition (\ref{eq:def-r-x}), 
$$
r(x_n) \to r(x)>0 \qquad \text{as $n \to \infty$,}
$$
there exists $n \in \N$ such that $x \in B_{r(x_n)}(x_n)$. Consequently, $x \in W$ by Claim 2. The proof is finished.
	\end{proof}
	
Finally, we complete the 

\begin{proof}[Proof of Theorem~\ref{main-theorem}]
We assume that $u \not \equiv 0$ in $\Omega$, and we let $\Omega' \subset \subset \Omega$ be an open set with $u \not \equiv 0$ in $\Omega'$. By assumption and Corollary~\ref{sec:supersolution-variational-supersolution}, the function $u \in \cV^J(\Omega')$ is a variational supersolution of the equation $Iu =c(x)u$ in $\Omega'$. Since also $u \ge 0$ in $\R^N$ by assumption and since (j\,\!1), (j\,\!2) imply (J1) -- (J3) for the kernel $(x,y) \mapsto J(x,y)= j(x-y)$, Theorem \ref{hopf-simple2-variant1} yields that $u$ is strictly positive in $\Omega'$. Since $\Omega$ is connected, it now follows that $u$ is strictly positive in $\Omega$. 
\end{proof}

\appendix 

\section{Appendix}\label{a}

In this section we give the proof of Lemmas~\ref{lattice-intersection} and \ref{lattice-path}. So in the following, we let $v_1,\ldots, v_N \in \R^N$ be linearly independent, and we let $G:=\sum \limits_{k=1}^{N}\Z v_k \subset \R^N$ be the corresponding lattice. 
We recall Lemma~\ref{lattice-intersection}.

\begin{lemma}\label{lattice-intersection-appendix}
Let $x_0 \in \R^N$ and $r>\frac{1}{2}\sum \limits_{k=1}^N|v_k|$. Then $B_r(x_0) \cap G\neq \emptyset$. 
\end{lemma}

\begin{proof}
Consider the fundamental domain $\Pi:= \bigl \{\sum \limits_{k=1}^N \alpha_k v_k \::\: \text{$-\frac{1}{2} \le \alpha_k \le \frac{1}{2}$ for $k=1,\dots,N$} \bigr\}.$ Since the translates $v + \Pi$, $v \in G$ cover the whole space $\R^N$, there exists $v \in G$ with $x_0 -v \in \Pi$. Hence $x_0 -v = \sum \limits_{k=1}^N \alpha_k v_k$ with some $\alpha_k \in [-\frac{1}{2},\frac{1}{2}]$, $k=1,\dots,N$ and therefore $|x_0 - v| \le \frac{1}{2} \sum \limits_{k=1}^N |v_k|<r$. Consequently, $v \in G \cap B_r(x_0)$, so that $G \cap B_r(x_0) \not = \varnothing.$ 
\end{proof}

We now turn to the proof of Lemma~\ref{lattice-path}. It will be convenient to use the following definition. 
 
\begin{defi}
\label{sec:appendix-1-def-path}
Let $A \subset \R^N$ be an arbitrary subset, and let $x, x' \in A$ such that $x-x' \in G$. In the following, a {\em $G$-path in $A$ from $x'$ to $x$} is defined as an ordered set of points $w_0,\ldots, w_n\in A$ such that $w_0=x'$, $w_n=x$ and 
$$
w_{\ell}-w_{\ell-1} \in G_*:=\{ \pm v_1,\dots, \pm v_N\} \qquad \text{for $\ell = 1,\dots,N$.}
$$
\end{defi}

With this definition, we may now reformulate Lemma~\ref{lattice-path} as follows.

\begin{lemma}\label{lattice-path-appendix}	
For every $\rho \ge \sum \limits_{k=1}^N |v_k|$ and every $x,x' \in G$ with $|x-x'|<\rho$ there exists a $G$-path in $B_{4^{N-1} \rho}(x')$ from $x'$ to $x$. 
\end{lemma}

\begin{proof}
Without loss, we may assume that $x'=0$ in the following. We consider the affine subspaces 
$$
\cW_0:= \{0\},\qquad \cW_j:=\sum \limits_{i=1}^{j} \R v_i , \qquad j=1,\dots,N
$$
and the sublattices 
$$
G_{j}:= \sum_{k=1}^{j}  \Z v_k = G \cap \cW_j,\qquad j=1,\dots,N.
$$
We then prove the following claim by induction on $j$.\\ 
{\em { Claim A:} Let $j \in \{1,\dots,N\}$. For every $\rho \ge \sum \limits_{k=1}^N |v_k|$ and every $x \in G_j$ with 
$$
|x|< \rho+ \frac{1}{2} \sum_{i=1}^{j-1}|v_j|\qquad \text{and}\qquad \dist(x,\cW_{j-1})< \rho
$$
there exists a $G$-path in $B_{4^{j-1} \rho}(0)$ from $0$ to $x$.}\\

This is true for $j=1$, since in this case $x= k v_1$ for some $k \in \Z$, $|x|= \dist(x,\cW_0) < \rho$, and there is an obvious one-dimensional $G$-path in $B_{\rho}(0)$ from $0$ to $x$.

We now assume that Claim A is true for some fixed $j \in \{1,\dots,N-1\}$, i.e. 
\begin{equation}
  \label{eq:induction-hypothesis}
\left\{
  \begin{aligned}
&\text{For every $\rho \ge \sum \limits_{k=1}^N |v_k|$ and every $y \in G_{j}$ with $|y|< \rho+ \frac{1}{2} \sum_{i=1}^{j-1}|v_i|$}\\    
&\text{ and $\dist(y,\cW_{j-1})< \rho$ there exists a $G$-path in $B_{4^{j-1} \rho}(0)$ from $0$ to $y$.}  
  \end{aligned}
\right. 
\end{equation}
We fix $\rho \ge \sum \limits_{i=1}^{N}|v_i|$, and we suppose by contradiction that Claim A is false for $j+1$ and this choice of $\rho$.
Then there exists $x= y + k v_{j+1} \in G_{j+1}$ with $y \in G_{j}$ and $k \in \Z$ such that 
\begin{equation}
  \label{eq:assumption-inequalities}
|x|< \rho+ \frac{1}{2} \sum_{i=1}^{j}|v_i|, \qquad \dist(x,\cW_{j})< \rho,
\end{equation}
and such that there does not exist a $G$-path in 
$B_{4^{j}\rho}(0)$ from $0$ to $x$. Without loss we may assume that $k \ge 0$, and that $k$ is chosen minimally with this property. 
In the case $k=0$ we have $x=y$ and thus 
\begin{equation*}
|x|< \rho+ \frac{1}{2} \sum_{i=1}^{j}|v_i| \le 2\rho , \qquad \dist(x,\cW_{j-1}) \le |x|< 2\rho,
\end{equation*}
so that by (\ref{eq:induction-hypothesis}) -- applied with $2 \rho$ in place of $\rho$ --there exists a $G$-path in $B_{4^{j-1} 2\rho}(0) \subset B_{4^j \rho}(0)$ from $0$ to $x$. This contradicts our choice of $x$, and thus we have $k>0$. 

Let $x_1=y+(k-1)v_{j+1}$, and let $x_1^* \in \cW_{j}$ be the orthogonal projection of $x_1$ on $\cW_{j}$, so that 
\begin{equation}
  \label{eq:orthogonal-ineq}
|x_1^*| \le  |x_1| \qquad \text{and}\qquad |x_1-x_1^*| =\dist(x_1,\cW_{j}) \le |x_1|.
\end{equation}
Since the sets $\bigl\{y_*+ \sum\limits_{i=1}^{j}\alpha_iv_i\;:\; -\frac{1}{2} \leq \alpha_i\leq \frac{1}{2}\text{ for $i=1,\ldots, j$ }\bigr\}$, $y_*\in G_{j}$
cover $\cW_{j}$, there exists $y^* \in G_{j}$ and $\alpha_i \in [-\frac{1}{2},\frac{1}{2}]$ with $x_1^* = y^* + \sum \limits_{i=1}^{j} \alpha_i v_i$. The point $x':=x_1- y^*$ then satisfies 
\begin{align}
\dist(x',\cW_{j})=\dist(x_1,\cW_{j})&= (k-1) \, \dist(v_{j+1},\cW_{j}) \nonumber\\
& \le k \,\dist(v_{j+1},\cW_{j})= \dist(x,\cW_{j}) < \rho \label{appendix-add-ineq}  
\end{align}
and, by (\ref{eq:orthogonal-ineq}) and (\ref{appendix-add-ineq}),
\begin{equation}
  \label{eq:estimate-x-prime}
|x'|= |x_1-y^*| \le |x_1-x_1^*| + |x_1^*-y^*| \le  \dist(x_1,\cW_{j}) + \frac{1}{2}\sum_{i=1}^{j}|v_i| < \rho +  \frac{1}{2}\sum_{i=1}^{j}|v_i|.
\end{equation}
By the minimality property of $k$, this implies the existence of a $G$-path $\Gamma_1$ from $0$ to $x'$ in $B_{4^{j}\rho}(0)$. Moreover, by (\ref{eq:assumption-inequalities}) and (\ref{eq:orthogonal-ineq}), 
$$
|y^*| \le |x_1^*| +  \frac{1}{2}\sum_{i=1}^{j}|v_i| \le |x_1| +  \frac{1}{2}\sum_{i=1}^{j}|v_i| \le |x|+|v_{j+1}| + \frac{1}{2}\sum_{i=1}^{j}|v_i| < \rho + \sum_{i=1}^{j+1}|v_j| \le 2 \rho,
$$
and thus also $\dist(y^*,\cW_{j-1}) < 2\rho$. Thus (\ref{eq:induction-hypothesis}) -- applied with $2 \rho$ in place of $\rho$ -- yields a $G$-path from $0$ to $y_*$ in $B_{4^{j-1}2\rho}(0)$. By mere translation, this path gives rise to $G$-path $\Gamma_2$ from $x' = x_1 - y^*$ to $x_1$ in $B_{4^{j-1}2\rho}(x')$, whereas 
$$ 
B_{4^{j-1}2\rho}(x') \subset B_{4^{j} \rho}(0) \qquad \text{since $|x'| \le 2 \rho$ by (\ref{eq:estimate-x-prime}).}   
$$
Composing $\Gamma_1$ and $\Gamma_2$, we then get a $G$-path from $0$ to $x_1$ in $B_{4^{j}\rho}(0)$. Simply adding $x=y+ k v_j$ as an endpoint and using that $x \in B_{4^{j}\rho}(0)$ by assumption, we finally obtain a $G$-path from $0$ to $x$ in $B_{4^{j}\rho}(0)$. This contradicts our assumption and finishes the proof of Claim A for $j+1$.

By induction, the proof of Claim A is thus finished. Applying Claim A with $j=N$ and noting that 
$\dist(x,\cW_{N-1}) \le |x|$ for every $x \in G$, we finally deduce the claim of the lemma. 
\end{proof}

\bibliographystyle{amsplain}

\begin{thebibliography}{10}

\bibitem{A09}
D.~Applebaum, \emph{L\'evy {P}rocesses and {S}tochastic {C}alculus}, Cambridge
University Press, Cambridge, 2009.

\bibitem{BBCK06} M.~T.~Barlow, R.~F.~Bass, Z.-Q.~Chen and M.~Kassmann, \emph{Non-local Dirichlet Forms and Symmetric Jump Processes},
        Trans. Amer. Math. Soc. \textbf{361.4} (2009), 1963--1999.

\bibitem{BB00} K.~Bogdan and T.~Byczkowski, \emph{Potential Theory of Schr\"odinger Operator based on fractional Laplacian},
Probab. Math. Statist. \textbf{20.2} (2000), 293--335.

\bibitem{CS14} X. Cabr{\'e}, Y. Sire, \emph{Nonlinear equations for fractional Laplacians, I: Regularity, maximum principles, and Hamiltonian estimates.} 
Ann. Inst. H. Poincaré Anal. Non Linéaire \textbf{31} (2014), 23--53.

\bibitem{CRS10}
L.~A.~Caffarelli, J.-M.~Roquejoffre and Y.~Sire, \emph{Variational problems for free boundaries for the fractional Laplacian}, J. Eur. Math. Soc. (JEMS) \textbf{12.5} (2010), 1151--1179.

\bibitem{DK15} B.~Dyda and M.~Kassmann, \emph{Regularity estimates for elliptic nonlocal operators}, preprint (2015), available online at \url{https://arxiv.org/pdf/1509.08320v2.pdf}.

\bibitem{DK16} B.~Dyda and M.~Kassmann, \emph{Function spaces and extension results for nonlocal Dirichlet problems}, preprint (2016), available online at \url{https://arxiv.org/pdf/1612.01628v1.pdf}.

\bibitem{NPV11}
E.~di~Nezza, G.~Palatucci, and E.~Valdinoci, \emph{Hitchhiker's {G}uide to the {F}ractional {S}obolev {S}paces},  Bull. Sci. Math. \textbf{136.5} (2012), 521--573.

\bibitem{EG92}
L.~C.~Evans and R.~F.~Gariepy, \emph{Measure theory and fine properties of functions}, {CRC Press, Boca Raton, FL}, 2015

\bibitem{FJ13}
M.~M.~Fall and S.~Jarohs, \emph{Overdetermined problems with fractional Laplacian}, ESAIM Control Optim. Calc. Var. \textbf{21.4} (2015), 924--938.

\bibitem{fall-weth}
M.~M.~Fall. and T.~Weth, \emph{Liouville theorems for a general class of nonlocal operators}, Potential Anal. \textbf{45} (2016), 187--200.

\bibitem{FK12}
M.~Felsinger and M.~Kassmann, \emph{Local regularity for parabolic nonlocal
  operators}, Comm. Partial Differential Equations \textbf{38.9} (2013), 1539--1573.

\bibitem{FKV13}
M.~Felsinger, M.~Kassmann and P.~Voigt, \emph{The Dirichlet problem for nonlocal operators},  Math. Z. \textbf{279} (2015), 779--809.

\bibitem{garcia-rossi}
J.~Garc\'\i a-Meli\'an and J.~D.~Rossi, \emph{Maximum and antimaximum principles for some nonlocal diffusion
	operators}, Nonlinear Anal. \textbf{71.12} (2009), 6116--6121.

\bibitem{GT}
 D.~Gilbarg and N.~S.~Trudinger, \emph{Elliptic partial differential equations of second order}, Springer-Verlag, Berlin, 2001.


\bibitem{grisvard} P.~Grisvard, \emph{Elliptic Problems in Nonsmooth Domains}, SIAM Classics in Applied Mathematics 69, Philadelphia, PA 2011.

\bibitem{J05}
N.~Jacob, \emph{Pseudo {D}ifferential {O}perators and {M}arkov {P}rocesses, {V}ol. {I}, {II}, {III}}, Imperial College Press, London, 2005.


\bibitem{JW14}
S.~Jarohs and T.~Weth, \emph{Symmetry via antisymmetric maximum principles in nonlocal problems of variable order}, Ann. Mat. Pura Appl. (4) \textbf{195.1} (2016), 273--291.

\bibitem{jarohs-thesis}
S.~Jarohs, \emph{Symmetry via maximum principles for nonlocal nonlinear boundary value problems}, doctoral thesis, 2015.

\bibitem{K11}
M.~Kassmann, \emph{{A new formulation of {H}arnack's inequality for nonlocal
  operators}}, C. R. Math. Acad. Sci. Paris \textbf{349} (2011), 637--640.


\bibitem{KM13} M.~Kassmann and A.~Mimica, \emph{Intrinsic scaling properties for nonlocal operators}, J. Eur. Math. Soc. \textbf{19.4} (2017), 983--1011.

\bibitem{MN17}
R.~Musina and A.~I.~Nazarov, \emph{Strong maximum principles for fractional Laplacians}, preprint (2017), available online at \url{https://arxiv.org/pdf/1612.01043v2.pdf}.
	
\bibitem{PW84}
   M.~H.~Protter and H.~F.~Weinberger, \emph{Maximum principles in differential equations}, Springer-Verlag, New York, 1984.

\bibitem{QM05}
Q.~Y.~Guan and Z.~M.~Ma, \emph{Boundary problems for fractional Laplacians}, Stoch. Dyn. \textbf{5.3} (2005), 385--424.


\bibitem{SR} M.~Reed and B.~Simon, {\em Methods of Modern Mathematical Physics: I Functional Analysis}, Academic Press, San Diego, 1980.

\bibitem{SV} R.~Servadei and E.~Valdinoci, \emph{Variational methods for non-local operators of elliptic type},  Discrete Contin. Dyn. Syst. \textbf{33.5} (2013), 2105--2137.

\bibitem{W89} W.~Walter, \emph{A theorem on elliptic differential inequalities with an application to gradient bounds}, Math. Z. \textbf{200.2} (1989), 293--299.

\end{thebibliography}

\end{document}